\documentclass[11pt, letterpaper]{amsart}
\pdfoutput=1

\usepackage[utf8]{inputenc}
\usepackage[T1]{fontenc}
\usepackage{microtype}

\usepackage{amssymb, amsfonts, amstext, amscd, amsrefs}
\usepackage{mathtools, relsize}
\usepackage{latexsym}
\usepackage{bbm}
\usepackage{comment}

\usepackage[pdftitle={The ideal intersection property and primality for reduced crossed products}]{hyperref}
\usepackage{color}
\definecolor{darkblue}{rgb}{0.0,0.0,0.3}
\hypersetup{colorlinks,breaklinks,linkcolor=red,urlcolor=red,anchorcolor=red,citecolor=red}

\usepackage{enumitem}

\usepackage{amsthm}
\numberwithin{equation}{section}

\theoremstyle{plain}
\newtheorem{thm}[equation]{Theorem}
\newtheorem{cor}[equation]{Corollary}
\newtheorem{lem}[equation]{Lemma}
\newtheorem{prop}[equation]{Proposition}
\newtheorem*{theorem*}{Theorem}
\newtheorem*{cor*}{Corollary}
\newtheorem*{prop*}{Proposition}
\newtheorem{thmintro}{Theorem}

\theoremstyle{definition}
\newtheorem{defn}[equation]{Definition}

\newtheorem{example}[equation]{Example}
\newtheorem{problemintro}{Problem}

\theoremstyle{remark}
\newtheorem{rem}[equation]{Remark}

\newcommand{\ZZ}{\mathbb{Z}}

\newcommand{\id}{\mathrm{id}}
\newcommand{\rC}{\mathrm{C}}

\author{Matthew Kennedy}
\address{Department of Pure Mathematics, University of Waterloo \\ 200 University Avenue West \\ Waterloo, Ontario, N2L 3G1 \\ Canada}
\email{matt.kennedy@uwaterloo.ca}

\author{Larissa Kroell}
\address{Department of Pure Mathematics, University of Waterloo \\ 200 University Avenue West \\ Waterloo, Ontario, N2L 3G1 \\ Canada}
\email{lkroell@uwaterloo.ca}

\author{Camila F. Sehnem}
\address{Department of Pure Mathematics, University of Waterloo \\ 200 University Avenue West \\ Waterloo, Ontario, N2L 3G1 \\ Canada}
\email{camila.sehnem@uwaterloo.ca}

\title[Primality and the ideal intersection property]{Primality and the ideal intersection property for reduced crossed products}

\thanks{The first and second author acknowledge the support of the Natural Sciences and Engineering Research Council of Canada (NSERC)}

\begin{document}

\begin{abstract}
We consider the ideal structure of reduced crossed products over discrete groups. First, we completely characterize primality for reduced crossed products. Second, we characterize the ideal intersection property for reduced crossed products over FC-hypercentral groups. Both of these characterizations are intrinsic, in terms of conditions on the underlying dynamics. A key intermediate result is a complete characterization of the regular ideal intersection property for reduced crossed products. For C*-dynamical systems over groups with restrictive subgroup structure, these characterizations simplify even further, which we demonstrate with a number of examples.
\end{abstract}

\maketitle

\setcounter{tocdepth}{1}

\tableofcontents

\section{Introduction}
This paper concerns the ideal structure of reduced crossed products over discrete groups. Let $(A,G,\alpha)$ be a C*-dynamical system consisting of a C*-algebra $A$, a discrete group $G$ and an action $\alpha$ of $G$ by automorphisms on $A$. The corresponding reduced crossed product $A \times_\lambda G$ is a C*-algebra generated by a copy of $A$ and a copy of the left regular representation $\lambda$ of $G$ satisfying
\[
\lambda_g a \lambda_g^* = \alpha_g(a), \quad g \in G,\ a \in A.
\]

The relationship between the ideal structure of the reduced crossed product and the structure of the underlying C*-dynamical system has been a focus point in the theory of C*-algebras for at least the past half-century. In this paper, we will consider two fundamental problems along these lines.

\begin{problemintro} \label{problemintro:primality}
Characterize the primality of reduced crossed products in terms of the underlying C*-dynamical system.
\end{problemintro}

The first problem that we will consider asks for a characterization of the primality of reduced crossed products. Recall that a C*-algebra is prime if every pair of nonzero ideals in the algebra has non-empty intersection. For separable C*-algebras, Dixmier \cite{Dixmier1960} proved that primality is equivalent to the existence of a faithful irreducible representation.

This problem was first considered in an influential paper of Rieffel \cite{Rieffel1980}. He obtained a characterization of primality for reduced crossed products of separable C*-dynamical systems over finite groups using ideas from Morita theory and imprimitivity theory. Subsequently, Hamana \cite{Hamana1985} obtained another characterization of primality for reduced crossed products of C*-dynamical systems over finite groups using the theory of injective envelopes. Later, Echterhoff \cite{Echterhoff1990} established an equivalence between primality and simplicity for reduced crossed products of minimal C*-dynamical systems over FC-hypercentral groups. Finally, in very recent work, Geffen and Ursu obtained a characterization of primality for reduced crossed products of minimal C*-dynamical systems.

The first main result in this paper is a complete resolution of Problem~\ref{problemintro:primality}. Specifically, we completely characterize the primality of reduced crossed products in terms of the underlying C*-dynamical system. A key idea underlying this result, and many of the other results in this paper, is a new theory of induction for C*-dynamical systems. In line with other notions of induction, this is a mechanism for constructing new C*-dynamical systems over a group from a given C*-dynamical system over a subgroup. There are also corresponding tensor product decompositions for an induced C*-dynamical system and its reduced crossed product. However, unlike other notions of induction, the induction  here is at the level of the injective envelope of the system. As a result, the corresponding tensor product decompositions are also necessarily at this level. A crucial point, which has important structural implications, is that the injective envelope of a C*-dynamical system, as well as the injective envelope of its reduced crossed product, may decompose into tensor products, despite the fact that the system and its reduced crossed product do not. 

It will also be important for our arguments that there is a complementary theory of imprimitivity, for detecting when a given C*-dynamical system is induced. If a C*-dynamical system $(A,G,\alpha)$ is induced, then it is induced from a regular C*-dynamical subsystem $(J,H,\beta)$, consisting of a subgroup $H \leq G$, a regular $H$-invariant ideal $J \trianglelefteq A$ and an action $\beta$ obtained by restricting $\alpha$ in the obvious way. This new theory of induction and imprimitivity provides a conceptual framework for relating the structure of an induced C*-dynamical system to the structure of the corresponding C*-dynamical subsystem.

We are now ready to state the first main result (see Section~\ref{sec:primality}).

\begin{thmintro}\label{thmintro:primeness} Let $(A, G, \alpha)$ be a prime C*-dynamical system. The following are equivalent:
\begin{enumerate}
\item $A \times_\lambda G$ is prime;
\item Whenever $(A,G,\alpha)$ is induced from a regular C*-dynamical subsystem $(J,H,\beta)$ with $r \in H \setminus \{e\}$ and a $C_H(r)$-invariant essential hereditary C*-subalgebra $B \subseteq J$ such that $\alpha_r|_B = \exp(\delta)$ for a $C_H(r)$-commuting *-derivation $\delta$ of $B$, then the $H$-conjugacy class $\operatorname{Cl}_H(r)$ of $r$ is infinite;
\item Whenever $(I(A),G,\alpha)$ is induced from a regular C*-dynamical subsystem $(J,H,\beta)$ with $r \in H \setminus \{e\}$ such that $\beta_r = \operatorname{Ad}(u_r)$ for a $C_H(r)$-invariant unitary $u_r \in J$, then the $H$-conjugacy class $\operatorname{Cl}_H(r)$ of $r$ is infinite;
\end{enumerate}
If $A$ is separable, then these conditions are additionally equivalent to:
\begin{enumerate}[resume]
\item Whenever $(A,G,\alpha)$ is induced from a regular C*-dynamical subsystem $(J,H,\beta)$ with $r \in H \setminus \{e\}$, a nonzero $r$-invariant ideal $I \trianglelefteq J$ and a unitary $u \in M(I)$ such that
\begin{enumerate}
\item $I \cap \alpha_h(I)$ is an essential ideal in both $I$ and $\alpha_h(I)$ for all $h \in C_H(r)$ and 
\item $C := \|\alpha_r|_I - \operatorname{Ad}(u)\|$ and $D := \sup_{h \in C_H(r)} \|\alpha_h(u) - u\|$ satisfy $2\sqrt{2 - \sqrt{4-C^2}} + D < \sqrt{2}$,
\end{enumerate}
then the $H$-conjugacy class $\operatorname{Cl}_H(r)$ of $r$ is infinite.
\end{enumerate}
\end{thmintro}

We want to emphasize that Conditions (2) and (4) in Theorem \ref{thmintro:primeness} are ``intrinsic,'' by which we mean that they are stated in terms of the underlying C*-dynamical system, and are possible to check in practice. In contrast, Condition (3) is ``extrinsic,'' since it refers to the (unique) C*-dynamical system obtained by extending the action to the injective envelope of the underlying C*-algebra.

For the special case of minimal C*-dynamical systems, the equivalence between Conditions (1) and (3) recovers one of the main results from \cite{GeffenUrsu2023}*{Theorem 6.2} (see Remark \ref{rem:recover-geffen-ursu-primality}). In the separable minimal case, they also proved the equivalence between Conditions (1) and (4). A key step in their proof of the latter result is a characterization of automorphisms of separable C*-algebras that are implemented, on an invariant ideal, by a unitary in the injective envelope that is invariant under a given group of automorphisms (see \cite{GeffenUrsu2023}*{Theorem 7.15}). Their characterization, in terms of approximation by inner automorphisms that are implemented by approximately invariant unitaries, can be seen as an equivariant analogue of Elliott's characterization of proper outerness for automorphisms of separable C*-algebras from \cite{Elliott1980}*{Theorem 2.3}. We will utilize this characterization in our proof of the equivalence between Conditions (1) and (4) in the general separable case.

The equivalence between Conditions (2) and (4) does not require separability. In order to prove this equivalence, we will establish another characterization of automorphisms of C*-algebras with the above-mentioned properties in terms of *-derivations. This characterization can be seen as an equivariant analogue of Olesen and Pedersen's characterization of proper outerness for automorphisms of C*-algebras from \cite{OlesenPedersen1982}*{Theorem 4.3}. See Section \ref{sec:intrinsic} for more details about these results.

In Section \ref{sec:examples}, we will apply Theorem \ref{thmintro:primeness} to a number of examples. We will show that for groups with restrictive subgroup structure, the conditions in the theorem can be significantly simplified.

\begin{problemintro} \label{problemintro:ideal-intersection-property}
Characterize the ideal intersection property for C*-dynamical systems.
\end{problemintro}

The second problem that we will consider asks for a characterization of the ideal intersection property for C*-dynamical systems. This problem was first considered by O'Donovan in \cite{ODonovan1975}, and Olesen and Pedersen in \cite{OlesenPedersen1980} (see Remark 4.8 of the latter paper). A C*-dynamical system is said to have the ideal intersection property if every nonzero ideal of the corresponding reduced crossed product has nonzero intersection with the underlying C*-algebra. Note that for minimal C*-dynamical systems, the ideal intersection property is equivalent to the simplicity of the reduced crossed product.

The ideal intersection property is perhaps the most fundamental example of the ``uniqueness'' phenomenon identified by Cuntz and Krieger \cite{CuntzKrieger1980}, whereby it is possible to infer a great deal of information about the structure of a C*-algebra from the structure of a well-understood, and typically much smaller, C*-subalgebra. Further, Sierakowski \cite{Sierakowski2010} proved that the existence of a lattice isomorphism between the ideals of the reduced crossed product and the invariant ideals of the underlying C*-algebra can be characterized in terms of the ideal intersection property. This is typically the only situation in which it is possible to completely describe the ideal structure of the reduced crossed product in terms of the underlying C*-dynamical system.

Olesen and Pedersen obtained a characterization of the ideal intersection property for C*-dynamical systems over abelian groups. Later, Kawamura and Tomiyama \cite{KawamuraTomiyama1990} obtained a characterization of the ideal intersection property for abelian C*-dynamical systems over amenable groups. This characterization was then extended to amenable abelian C*-dynamical systems by Archbold and Spielberg \cite{ArchboldSpielberg1994}. Recently, Kawabe \cite{Kawabe2017} obtained a characterization of the ideal intersection property for unital abelian C*-dynamical systems, generalizing a characterization of C*-simple groups by the first author \cite{Kennedy2020}. Subsequently, Kalantar and Scarparo \cite{KalantarScarparo2022} extended Kawabe's result to non-unital abelian C*-dynamical systems.

For non-abelian C*-dynamical systems, the ideal intersection property has proven more difficult to understand. However, significant progress has recently been made. The first author and Schafhauser \cite{KennedySchafhauser2019} identified a cohomological obstruction to the ideal intersection property. Under the assumption that this obstruction vanishes, they obtained a characterization of the ideal intersection property for arbitrary C*-dynamical systems. Subsequently, Geffen and Ursu \cite{GeffenUrsu2023} obtained a characterization of the ideal intersection property in two important special cases: for C*-dynamical systems over FC (finite conjugacy class) groups and minimal C*-dynamical systems over FC-hypercentral groups.

The second main result in this paper is a resolution of Problem \ref{problemintro:ideal-intersection-property} for FC-hypercentral groups. Specifically, we characterize the ideal intersection property for C*-dynamical systems over FC-hypercentral groups. A group is FC-hypercentral if it has no non-trivial ICC (infinite conjugacy class) quotient. This is a large class of groups that contains every finitely generated group of polynomial growth. For prime C*-dynamical systems over FC-hypercentral groups, our result implies an equivalence between the primality of the corresponding reduced crossed product and the ideal intersection property for the system, extending a result of Echterhoff \cite{Echterhoff1990} in the minimal case. We will also demonstrate that our characterization of the ideal intersection property does not extend beyond the class of FC-hypercentral groups (see Example \ref{ex:non-fc-case}). This is in line with recent results establishing a dichotomy between the dynamical properties of FC-hypercentral groups and non-FC-hypercentral groups (see \cites{FrischHartmanTamuzVahidiFerdowsi2019,FrischTamuzVahidiFerdowsi2019}). 

Our characterization of the ideal intersection property for C*-dynamical systems over FC-hypercentral groups also requires the induction and imprimitivity theory developed in this paper. We will say that a C*-dynamical system $(A,G,\alpha)$ is sub-induced by a regular C*-dynamical system $(J,H,\beta)$ if the restriction of $(A,G,\alpha)$ to a regular $G$-invariant ideal of $A$ containing the regular ideal $J$ is induced by $(J,H,\beta)$ in the above sense.

We are now ready to state the second main result (see Section~\ref{sec:intersection-property}).

\begin{thmintro} \label{thmintro:FC-hypercentral-ideal-intersection-property}
Let $(A,G,\alpha)$ be a unital C*-dynamical system such that $G$ is FC-hypercentral. The following are equivalent:
\begin{enumerate}
\item $(A,G,\alpha)$ has the ideal intersection property;

\item Whenever $(A,G,\alpha)$ is sub-induced from a regular C*-dynamical subsystem $(J,H,\beta)$ with $r \in H \setminus \{e\}$ and a $C_H(r)$-invariant essential hereditary C*-subalgebra $B \subseteq J$ such that $\alpha_r|_B = \exp(\delta)$ for a $C_H(r)$-commuting *-derivation $\delta$ of $B$, then the $H$-conjugacy class $\operatorname{Cl}_H(r)$ of $r$ is infinite;

\item Whenever $(I(A),G,\alpha)$ is sub-induced from a regular C*-dynamical subsystem $(J,H,\beta)$ with $r \in H \setminus \{e\}$ such that $\beta_r = \operatorname{Ad}(u_r)$ for a $C_H(r)$-invariant unitary $u_r \in J$, then the $H$-conjugacy class $\operatorname{Cl}_H(r)$ of $r$ is infinite;

\end{enumerate}
If, in addition, $A$ is separable, then these conditions are equivalent to:
\begin{enumerate}[resume]
\item Whenever $(A,G,\alpha)$ is sub-induced from a regular C*-dynamical subsystem $(J,H,\beta)$ with $r \in H \setminus \{e\}$, a nonzero $r$-invariant ideal $I \trianglelefteq J$ and a unitary $u \in M(I)$ such that
\begin{enumerate}
\item $I \cap \alpha_h(I)$ is an essential ideal in both $I$ and $\alpha_h(I)$ for all $h \in C_H(r)$ and 
\item $C := \|\alpha_r|_I - \operatorname{Ad}(u)\|$ and $D := \sup_{h \in C_H(r)} \|\alpha_h(u) - u\|$ satisfy $2\sqrt{2 - \sqrt{4-C^2}} + D < \sqrt{2}$,
\end{enumerate}
then the $H$-conjugacy class $\operatorname{Cl}_H(r)$ of $r$ is infinite.
\end{enumerate}
\end{thmintro}

As with Theorem \ref{thmintro:primeness}, we want to emphasize that Conditions (2) and (4) in Theorem \ref{thmintro:FC-hypercentral-ideal-intersection-property} are intrinsic, in terms of the underlying C*-dynamical system and possible to check in practice, while Condition (3) is extrinsic, in terms of the injective envelope of the underlying C*-algebra. In fact, it will only be possible to establish the equivalence between Condition (1) and Conditions (2) and (4) once we have established the equivalence between Conditions (1) and (3), and both of these steps will require a significant amount of work.

This strategy, of studying operator-algebraic structures in terms of their interplay with boundary-theoretic structures, which we refer to as noncommutative boundary theory, will be utilized throughout this paper. For C*-dynamical systems, the starting point for this theory is Hamana's \cite{Hamana1985} result that every C*-dynamical system has a unique minimal injective extension called the injective envelope. The injective envelope is typically a much larger C*-dynamical system, but it has a number of nice properties and, crucially, it encodes a great deal of information about the original system. Kalantar and the first author \cite{KalantarKennedy2017} proved that this extension is a noncommutative generalization of Furstenberg's universal topological boundary of a group, and applied these ideas to characterize C*-simple groups. These ideas have subsequently been extended and applied with a great deal of success to the study of reduced crossed products. For example, they feature prominently in \cites{BryderKennedy2018, Kawabe2017,KennedySchafhauser2019,GeffenUrsu2023}.

It can be quite difficult to reduce extrinsic conditions on boundary-theoretic objects to intrinsic conditions on the objects of interest, and techniques for doing this have become a major component of noncommutative boundary theory. For example, consider the case of a discrete group $G$, which corresponds to the trivial C*-dynamical system $(\mathbb{C},G,\alpha)$. The corresponding boundary-theoretic object is the injective envelope, which is the C*-dynamical system $(\rC(\partial_F G), G, \alpha)$, where $\partial_F G$ denotes the Furstenberg boundary of $G$, i.e. the universal minimal strongly proximal $G$-flow. The main result in \cite{KalantarKennedy2017} is an extrinsic characterization of the C*-simplicity of $G$, i.e. the simplicity of the reduced C*-algebra $\mathrm{C}^*_\lambda(G)$, which is the reduced crossed product of the C*-dynamical system $(\mathbb{C},G,\alpha)$: $G$ is C*-simple if and only if the $G$-action on $\partial_F G$ is free. On the other hand, the main result in \cite{Kennedy2020} is an intrinsic characterization of C*-simplicity of $G$: $G$ is C*-simple if and only if it has no amenable confined subgroups.

In the present paper, we will need to perform this reduction from extrinsic to intrinsic conditions for arbitrary C*-dynamical systems. For this, we will utilize the new theory of induction and imprimitivity mentioned above.
The following result about tensor product decompositions will play a central role in this paper. This was inspired by Rieffel's approach to the primality of reduced crossed products of C*-dynamical systems over finite groups \cite{Rieffel1980}, and by Vaes' characterization of the factoriality of abelian W*-dynamical systems \cite{Vaes2020}. It is an analogue of Green's imprimitivity theorem \cite{Green1980}*{Theorem 2.13} for this new notion of induction. See Section \ref{sec:imprimitivity} for details.

\begin{thmintro} \label{introthm:imprimitivity-reduced-crossed-product}
Let $(A,G,\alpha)$ be a C*-dynamical system. Suppose that $(A,G,\alpha)$ is induced from a regular C*-dynamical subsystem $(J,H,\beta)$. Then
\[
I(A) \cong \ell^\infty(G/H) \overline{\otimes} I(J)
\]
and
\[
I(A \times_\lambda G) \cong \mathcal{B}(\ell^2(G/H)) \overline{\otimes} I(J \times_\lambda H),
\]
where the tensor products on the right hand side are the monotone complete tensor products.
\end{thmintro}

An important intermediate step towards our main results that may be of independent interest is a complete characterization of the regular ideal intersection property (see Section \ref{sec:regular-ideal-intersection-property}). A C*-dynamical system is said to have the regular ideal intersection property if every nonzero regular ideal of the corresponding reduced crossed product has nonzero intersection with the underlying C*-algebra. This property was recently introduced by Brown, Fuller, Pitts and Reznikoff \cite{BrownFullerPittsReznikoff2024} (see also \cite{Exel2023}) in the general setting of C*-inclusions. Although it is not true in general, it turns out that for FC-hypercentral groups, the regular ideal intersection property coincides with the ideal intersection property, and the proof of this fact is a key step in the proof of Theorem \ref{thmintro:FC-hypercentral-ideal-intersection-property}. See Section \ref{sec:regular-ideal-intersection-property} for our characterization of the regular ideal intersection property. 

In addition to this introduction, there are 9 sections in this paper. In Section \ref{sec:preliminaries}, we will review some essential facts about C*-algebras, C*-dynamical systems and reduced crossed products, including injective envelopes. In Section \ref{sec:imprimitivity}, we will introduce the new theory of induction and imprimitivity for C*-dynamical systems mentioned above. In Section \ref{sec:projections}, we will introduce an important technical device called a meandering projection that is useful for detecting the imprimitivity of a C*-dynamical subsystem. In Section \ref{sec:intrinsic}, we will prove the equivalence of the extrinsic and intrinsic conditions in our main results. In Section \ref{sec:regular-ideal-intersection-property}, we will establish our characterization of the regular ideal intersection property, and in Section \ref{sec:primality}, we will establish our characterization of primality for reduced crossed products. In Section \ref{sec:pseudo-expectations}, we will obtain some technical results about an important class of maps on reduced crossed products called pseudo-expectations that will be required for Section \ref{sec:intersection-property}, where we will establish our characterization of the ideal intersection property for FC-hypercentral groups. Finally, in Section \ref{sec:examples}, we will present some examples and applications of our main results.

\section*{Acknowledgements} The first author is grateful to Dan Ursu for helpful discussions about \cite{GeffenUrsu2023}. 

\section{Preliminaries} \label{sec:preliminaries}

In this section we will briefly review some of the basic facts about C*-algebras, C*-dynamical systems and reduced crossed products that we will require. Because they will be utilized throughout this paper, we will also review some essential facts about injective and monotone complete C*-algebras, including injective envelopes of C*-dynamical systems. For a general reference on C*-algebras, we refer the reader to Blackadar's book \cite{Blackadar2006}, and for a general reference on C*-dynamical systems and crossed products, we refer the reader to Williams' book \cite{Williams2007}. Since the results we will require about injective and monotone complete C*-algebras have not yet been collected in one place, we will provide references along the way. However, the book of Paulsen \cite{Paulsen2002} provides an introduction to operator space theory, including injective envelopes, and the book of Sait\^{o} and Wright \cite{SaitoWright2015} provides an introduction to the theory of monotone complete C*-alegbras.

\subsection{C*-algebras, hereditary C*-subalgebras, ideals and primality}

Let $A$ be a C*-algebra. We will not generally assume that $A$ is unital. A C*-subalgebra $B \subseteq A$ is {\em hereditary} if whenever $a \in A$ and $b \in B$ satisfy $0 \leq a \leq b$, then $a \in B$. 

A (closed two-sided) ideal $J \trianglelefteq A$ is an example of a hereditary C*-subalgebra of $A$. Recall that the ideal $J$ is {\em essential} in $A$ if $J \cap I \ne 0$ for every nonzero ideal $I \trianglelefteq A$. We will be particularly interested in hereditary C*-subalgebras that generate essential ideals.

We will write $\tilde{A}$ for the minimal unitization of $A$, which contains $A$ as an essential ideal. We will also write $M(A)$ for the multiplier C*-algebra of $A$, which contains $A$ as an essential ideal. Note that if $A$ is already unital, then $A = \tilde{A} = M(A)$.

We will need to deal with hereditary C*-subalgebras that are not necessarily ideals when we consider *-derivations. It will be convenient to introduce the following definition, which generalizes the definition of an essential ideal. 

\begin{defn}
Let $A$ be a C*-algebra. We will say that a hereditary C*-subalgebra $B \subseteq A$ is {\em essential} in $A$ if the ideal generated by $B$ in $A$ is essential.
\end{defn}

\begin{defn}
A C*-algebra is {\em prime} if every pair of nonzero ideals in the C*-algebra has nonzero intersection.
\end{defn}

In other words, a C*-algebra $A$ is prime if and only if every nonzero ideal in $A$ is essential. This is equivalent to the zero ideal being a prime ideal in $A$. Note that every simple C*-algebra is prime. The primality of $A$ is equivalent to the primality of $\tilde{A}$, and to the primality of $M(A)$. It is an easy exercise to verify that a prime C*-alegbra necessarily has trivial center.

It is easy to see that a C*-algebra with a faithful irreducible representation is necessarily prime, and as mentioned in the introduction, Dixmier \cite{Dixmier1960} proved that the converse holds for separable C*-algebras. However, Weaver \cite{Weaver2003} showed that the converse does not necessarily hold for non-separable C*-algebras.

\subsection{Glimm ideals and the Glimm space}

Let $A$ be a C*-algebra. We will write $Z(A)$ for the center of $A$, which is an abelian C*-algebra. The {\em Glimm space} $\operatorname{Glimm}(A)$ of $A$ is the spectrum of $Z(A)$. An ideal of $A$ is a {\em Glimm ideal} if it can be written as $A K_x$ for a point $x$ in $\operatorname{Glimm}(A)$, where $K_x \trianglelefteq Z(A)$ denotes the corresponding maximal ideal
\[
K_x = \{f \in Z(A) : f(x) = 0 \}.
\]
Note that $A K_x$ is a closed two-sided ideal of $A$ by Cohen's factorization theorem.

For a Glimm ideal $A K_x$ of $A$ as above, it follows from the Dauns-Hoffman theorem that $Z(A) \cap A K_x = K_x$. This gives a bijective correspondence between points in $\operatorname{Glimm}(A)$ and Glimm ideals of $A$. This correspondence played an important role in \cite{KennedySchafhauser2019}, and it will also play an important role in this paper.

\subsection{Monotone complete and injective C*-algebras} \label{subsec:injective-envelopes}

As mentioned in the introduction, noncommutative boundary theory will play a major role throughout this paper. For this, we will require some essential facts about monotone complete and injective C*-algebras, including injective envelopes.

A C*-algebra $A$ is {\em monotone complete} if every norm bounded upward directed subset $X$ of self-adjoint elements in $A$ has a least upper bound, which we will denote by $\sup_A X$. A monotone complete C*-algebra is always unital (see e.g. \cite{SaitoWright2015}*{Proposition 2.1.15}). If $A$ is monotone complete and $G \subseteq \operatorname{Aut}(A)$ is a subgroup of automorphisms of $A$, then it is easy to see that the fixed point C*-subalgebra $A^G$ is also monotone complete. In particular, applying this to the subgroup of inner automorphisms of $A$ yields that the center $Z(A)$ of $A$ is monotone complete.

Consider the category consisting of unital operator systems with unital complete order homomorphisms as morphisms. A morphism is an {\em embedding} if it is a complete order embedding, and an {\em isomorphism} if it is an invertible embedding. An operator system $R$ is {\em injective} if whenever $S,T$ are operator systems with an embedding $\phi : S \to T$ and a morphism $\psi : S \to R$, then there is a morphism $\theta : T \to R$ extending $\psi$, meaning that $\theta \circ \phi = \psi$.

A unital C*-algebra $A$ is {\em injective} if it is injective in the category of unital operator systems, or equivalently, if its image under the forgetful functor from the category of unital C*-algebras to the category of unital operator systems is injective. It follows from a result of Choi and Effros \cite{ChoiEffros1977}*{Theorem~3.1} that every injective unital operator system is the image of an injective unital C*-algebra under this functor. An injective C*-algebra is always monotone complete (see e.g. \cite{SaitoWright2015}*{Theorem 8.1.21}). For abelian C*-algebras, monotone completeness is equivalent to injectivity (see e.g. \cite{SaitoWright2015}*{Theorem 2.3.7} and Lemma \ref{lem:equivalence-abelian-injectivity} at the end of this section).

An embedding $\iota : R \to T$ of operator systems is {\em essential} if whenever $T$ is an operator system and $\phi : S \to T$ is a morphism such that $\phi \circ \iota$ is an embedding, then $\phi$ is an embedding. 

Hamana \cite{Hamana1979} proved the existence and uniqueness of injective envelopes in the category of operator systems. Specifically, he proved that for an operator system $R$, there is a unique pair $(I(R),\iota)$ consisting of an injective operator system $I(R)$ and an essential embedding $\iota : R \to I(R)$. Furthermore, if $R$ is a C*-algebra, then $\iota$ is a *-homomorphism. The operator system $I(R)$ is the injective envelope of $R$, and we will identify $R$ with its image in $I(R)$. From above, $I(R)$ is a unital injective C*-algebra. 

The inclusion $R \subseteq I(R)$ is {\em rigid}, meaning that whenever $\phi : I(R) \to I(R)$ is a morphism such that $\phi|_R = \id_R$, then $\phi = \id$. It follows easily from this that every automorphism of $R$ has a unique extension to an automorphism of $I(R)$.

If $A$ is a non-unital C*-algebra, the injective envelope is defined by $I(A) := I(\tilde{A})$, where $\tilde{A}$ denotes the minimal unitization of $A$. Every automorphism of $A$ has a unique extension to an automorphism of $\tilde{A}$, and hence from above, a unique extension to an automorphism of $I(A)$. We will make frequent use of this fact.

Let $A$ be a monotone complete C*-algebra. For a projection $p \in A$, there is a smallest central projection $c(p) \in Z(A)$ called the {\em central cover} of $p$. In fact, $c(p) = \vee_{u \in U(A)} upu^*$. The ideal $Ac(p)$ is the smallest ideal of $A$ containing the hereditary C*-algebra $pAp$.

The following result is very well known. See \cite{HadwinPaulsen2011}*{Theorem 2.4} for a proof.

\begin{lem} \label{lem:equivalence-abelian-injectivity}
Let $A$ be a unital abelian C*-algebra. The following are equivalent:
\begin{enumerate}
\item $A$ is injective in the category of unital operator systems with unital complete order homomorphisms as morphisms;
\item $A$ is injective in the category of unital abelian C*-algebras with unital *-homomorphisms as morphisms;
\item $A$ has extremally disconnected spectrum.
\end{enumerate}
\end{lem}

Hamana \cites{Hamana1982a,Hamana1982b} introduced a notion of {\em monotone complete tensor product}. For a W*-algebra $M$ and a monotone complete C*-algebra $A$, the monotone complete tensor product $M \overline{\otimes} A$ is a monotone complete C*-algebra that is uniquely determined by $M$ and $A$. There are C*-algebra embeddings $M \otimes 1, 1 \otimes A \subseteq M \overline{\otimes} A$.

\subsection{Regular ideals}

Let $A$ be a C*-algebra. We will write $\operatorname{Id}(A)$ for the set of all ideals of $A$, and $\operatorname{Prim}(A)$ for the set of all primitive ideals of $A$. 

For a subset $X \subseteq A$, the corresponding annihilator $X^\perp$ is defined by
\[
X^\perp = \{a \in A : ax = xa = 0 \text{ for all } x \in X\}.
\]
An ideal $J \trianglelefteq A$ is said to be {\emph regular} if $J = J^{\perp \perp}$. We will write $\operatorname{Reg}(A)$ for the set of regular ideals of $A$.

For $X \subseteq A$ and $Y \subseteq \operatorname{Prim}(A)$, let
\[
\operatorname{hull}(X) = \{P \in \operatorname{Prim}(A) : X \subseteq P\}, \qquad \operatorname{ker}(Y) = \bigcap_{P \in Y} P.
\]
Writing $\overline{Y} = \operatorname{hull} \operatorname{ker}(Y)$, the map $Y \to \overline{Y}$ defines a Kuratowski closure operator on $\operatorname{Prim}(A)$. We equip $\operatorname{Prim}(A)$ with the corresponding topology, called the {\em hull-kernel} topology.

The map $J \to \operatorname{Prim}(A) \setminus \operatorname{hull}(J)$ is a bijection between ideals of $A$ and open subsets of $\operatorname{Prim}(A)$. An ideal $J \trianglelefteq A$ is regular if and only if the corresponding open subset $U$ is regular open, meaning that $\operatorname{int}\operatorname{cl}(U) = U$. 

There is a natural complete boolean algebra structure on the regular open subsets of $\operatorname{Prim}(A)$. Namely, join, meet and complementation are defined by
\[
\vee_i U_i = \operatorname{int} \operatorname{cl}(\cup_i U_i), \quad \wedge_i U_i = \operatorname{int}(\cap_i U_i), \quad \neg U = \operatorname{int}(\operatorname{Prim}(A) \setminus U).
\]
for regular open subsets $U,U_i \subseteq \operatorname{Prim}(A)$. This induces a boolean algebra structure on $\operatorname{Reg}(A)$. It is easy to verify that join, meet and complementation are given by
\[
\vee_i J_i = (\cup_i J_i)^{\perp \perp}, \quad \wedge_i J_i = \cap_i J_i, \quad \neg J = J^\perp,
\]
for $J,J_i \in \operatorname{Reg}(A)$.

Hamana \cite{Hamana1982c}*{Lemma 1.3} proved a very useful characterization of regular ideals in terms of the injective envelope $I(A)$ of $A$. Namely, an ideal $J \trianglelefteq A$ is regular if and only if there is a central projection $p \in Z(I(A))$ such that $J = A \cap Ap$.

If $B$ is an injective C*-algebra, so that $B = I(B)$, then an ideal $K \trianglelefteq B$ is regular if and only if there is a central projection $q \in Z(B)$ such that $K = Bq$. It follows that there is a bijective correspondence between regular ideals in an arbitrary C*-algebra $A$ and regular ideals in its injective envelope $I(A)$. For a regular ideal $J \trianglelefteq A$, the corresponding regular ideal in $I(A)$ is $I(J)$. In particular, regular ideals in injective C*-algebras are themselves injective. 

Recall that $Z(I(A))$ is an injective commutative C*-algebra, meaning in particular that it has extremally disconnected spectrum. In an extremally disconnected space, a subset is regular open if and only if it is clopen. Hence there is a bijective correspondence between regular ideals of $I(A)$ and clopen subsets of the spectrum of $Z(I(A))$.

We will require the following basic result about the orthogonality of regular ideals.

\begin{lem} \label{lem:orthogonal-regular-ideals}
Let $A$ be a C*-algebra and let $J,K \trianglelefteq A$ be regular ideals. Then $J$ and $K$ are orthogonal if and only if the corresponding regular ideals $I(J),I(K) \trianglelefteq I(A)$ are orthogonal. 
\end{lem}

\begin{proof}
Since $J,K$ are regular ideals, there are central projections $p,q \in Z(I(A))$ such that $J = A \cap I(A)p$ and $K = A \cap I(A)q$. If $I(J)$ and $I(K)$ are orthogonal, then it is clear that $J$ and $K$ are orthogonal. Conversely, suppose that $J$ and $K$ are orthogonal. Note that $I(J) \cap I(K) = I(A)r$, where $r \in Z(I(A))$ is the central projection $r = p \wedge q$. If $r \ne 0$, then $0 \ne A \cap I(A)r = J \cap K$ by the essentiality of the inclusion $A \subseteq I(A)$, which would contradict the fact that $J$ and $K$ are orthogonal. Hence $r = 0$, so $I(J) \cap I(K) = 0$ and $I(J)$ and $I(K)$ are orthogonal.
\end{proof}

\subsection{Derivations}

We will utilize the theory of *-derivations of C*-algebras. For a reference, we direct the reader to Pedersen's book \cite{Pedersen1979}*{Section 8.6}.

Let $A$ be a C*-algebra. A linear map $\delta : A \to A$ is a {\em derivation} on $A$ if
\[
\delta(ab) = a\delta(b) + \delta(a)b, \quad a,b \in A.
\]
If, in addition, $\delta$ preserves self-adjointness, then $\delta$ is said to be a {\em *-derivation}.

A derivation $\delta$ on $A$ is {\em inner} if there is $d \in A$ such that $\delta = \operatorname{ad}(d)$, where $\operatorname{ad}(d)(a) = da-ad$ for $a \in A$. Note that $\delta$ is a *-derivation if and only if $d = ic$ for some self-adjoint $c \in A$.

\subsection{Group theory}

We will restrict our attention to discrete groups in this paper, although we will typically not require any assumptions on their cardinality. Let $G$ be a discrete group. For an element $g \in G$, we will write $C_G(g)$ for the {\em centralizer subgroup} of $g$ in $G$,
\[
C_G(g) = \{h \in G : hgh^{-1} = g \},
\]
and we will write $\operatorname{Cl}_G(g)$ for the {\em conjugacy class} of $g$ in $G$,
\[
\operatorname{Cl}_G(g) = \{hgh^{-1} : h \in G\}.
\]
We will write $FC(G)$ for the FC-center of $G$, which is the normal subgroup of $G$ consisting of all elements in $G$ with a finite conjugacy class. If $G = FC(G)$, then $G$ is said to be {\em FC}.

The {\em FC-central series} of $G$ is obtained in the following way: Let $F_1 = FC(G)$, and for an ordinal number $i$, define $F_i$ by transfinite induction. If $i+1$ is a successor ordinal, then define $F_{i+1}$ to be the unique normal subgroup of $G$ such that $F_{i+1}/F_i = FC(G/F_i)$. If $j$ is a limit ordinal, then $F_j$ is the normal subgroup of $G$ defined by $F_j = \cup_{i < j} F_i$.  

We will write $FCH(G)$ for the {\em FC-hypercenter} of $G$, which is the normal subgroup $FCH(G) = \cup_i F_i$. If $G = FCH(G)$, then $G$ is said to be {\em FC-hypercentral}. The class of FC-hypercentral groups is quite large. For example, it follow from Gromov's characterization of finitely generated virtually nilpotent groups that it contains all finitely generated groups of polynomial growth. It is well known that $G$ is FC-hypercentral if and only if it has no non-trivial ICC quotients.

\subsection{C*-dynamical systems and reduced crossed products}
A {\em C*-dynamical system} is a triple $(A,G,\alpha)$ consisting of a C*-algebra $A$, a discrete group $G$ and an action $\alpha : G \to \operatorname{Aut}(A)$. For $g \in G$, we will write $\alpha(g)$ as $\alpha_g$. If $A$ is non-unital, then the action $\alpha$ has a unique extension to an action on the minimal unitization $\tilde{A}$ that we will also denote by $\alpha$, yielding the C*-dynamical system $(\tilde{A},G,\alpha)$. We will write $A^G$ for the $G$-fixed point C*-subalgebra of $A$,
\[
A^G = \{a \in A : \alpha_g(a) = a \text{ for all } g \in G \}.
\]

If $A$ is abelian, then letting $X$ denote the spectrum of $A$, $X$ is a locally compact Hausdorff space and $A \cong \rC_0(X)$. It follows from Gelfand duality that every C*-dynamical system $(A,G,\alpha)$ arises from a $G$-action $G \curvearrowright X$ via
\[
\alpha_g(f)(x) = f(g^{-1}x), \quad g \in G,\ f \in \rC_0(X),\ x \in X.
\]
Conversely, every $G$-action $G \curvearrowright X$ on a locally compact Hausdorff space $X$ gives rise to a C*-dynamical system $(\rC_0(X),G,\alpha)$.

From a C*-dynamical system $(A,G,\alpha)$, we can construct the corresponding reduced crossed product $A \times_{\alpha,\lambda} G$. This is the C*-algebra generated by a faithful representation of $A$ and a copy of the left regular representation $\lambda : G \to U(\ell^2(G))$ of $G$ such that, identifying $A$ with its image under this representation,
\[
\lambda_g a \lambda_g^* = \alpha_g(a),\quad g \in G,\ a \in A.
\]
Since the corresponding action $\alpha$ will always be clear from the context, we will suppress $\alpha$ and write $A \times_{\alpha,\lambda} G$ as $A \times_\lambda G$.

A characteristic property of the reduced crossed product is the existence of a faithful conditional expectation $E : A \times_\lambda G \to A$ satisfying
\[
E(a \lambda_g) = \delta_e(g) a,\quad a \in A,\ g \in G,
\]
where $\delta_e$ denotes the indicator function of the identity element $e \in G$. An element $b \in A \times_\lambda G$ is uniquely determined by its Fourier series
\[
b \sim \sum_{g \in G} a_g \lambda_g,
\]
where $a_g = E(b \lambda_g^*)$.

\begin{defn}
A C*-dynamical system $(A,G,\alpha)$ is {\em prime} if for every pair of nonzero $G$-invariant ideals $I,J \triangleleft A$, $I \cap J \ne \emptyset$.
\end{defn}

By \cite{Hamana1985}*{Proposition 6.4}, a C*-dynamical system $(A,G,\alpha)$ is prime if and only if every essential extension is prime. In particular, $(A,G,\alpha)$ is prime if and only if $(I(A),G,\alpha)$ is prime, or equivalently, if and only if $(I_G(A),G,\alpha)$ is prime. Note that an abelian C*-dynamical system $(\rC_0(X),G,\alpha)$ is prime if and only if every non-empty $G$-invariant open subset of $X$ is dense.

We will also require the definition of proper outerness, both for automorphisms and for C*-dynamical systems. For a comprehensive discussion of proper outerness, we refer the reader to~\cite{KwasniewskiMeyer2018}. The following definition is not standard, but it is equivalent to the definition given in \cite{Hamana1985}, and convenient for our purposes.

\begin{defn}
Let A be a C*-algebra. We will say that an automorphism $\alpha$ of $A$ is {\em quasi-inner} if there is a unitary $u \in I(A)$ such that $\alpha = \operatorname{Ad}(u)|_A$. We will say that $\alpha$ is {\em properly outer} if there is no non-trivial invariant regular ideal $J \trianglelefteq A$ such that the restriction $\alpha|_J$ is quasi-inner. We will say that a C*-dynamical system $(A,G,\alpha)$ is {\em properly outer} if for every $g \in G$, the corresponding automorphism $\alpha_g$ is properly outer.
\end{defn}

It is important to point out that there are intrinsic characterizations of properly outer automorphisms on a C*-algebra, i.e. characterizations not requiring any knowledge of the injective envelope (see \cite{Hamana1985}*{Theorem 7.4} and \cite{OlesenPedersen1982}*{Theorem 6.6}). For commutative C*-dynamical systems, proper outerness is equivalent to topological freeness.

It follows from \cite{KennedySchafhauser2019}*{Corollary~6.7} and \cite{Zarikian2019}*{Theorem~4.6} that if $(A,G,\alpha)$ is propertly outer, then $(A,G,\alpha)$ has the ideal intersection property (see Section \ref{sec:intersection-property}).

\subsection{Injective envelopes of C*-dynamical systems}

In order to introduce injective envelopes of C*-algebras in Section \ref{subsec:injective-envelopes}, it was necessary to discuss the category of unital operator systems. Similarly, in order to discuss injective envelopes of C*-dynamical systems, we will need to consider a category of dynamical systems where objects being acted on are unital operator systems.

An {\em nc (noncommutative) dynamical system} is a triple $(R,G,\alpha)$ consisting of a unital operator system $R$, a discrete group $G$ and an action $\alpha : G \to \operatorname{Aut}(R)$. For a fixed discrete group $G$, consider the category of nc dynamical systems with $G$-equivariant unital complete order homomorphisms as morphisms. A morphism $\phi : (R,G,\alpha) \to (S,G,\beta)$ is an {\em embedding} if it is a unital complete order embedding. 

An  {\em extension} of an nc dynamical system $(R,G,\alpha)$ is a pair $((S,G,\beta), \iota)$ consisting of an nc dynamical system $(S,G,\beta)$ and an embedding $\iota : (R,G,\alpha) \to (S,G,\beta)$. Note that if $\iota$ is the inclusion map, meaning that $R \subseteq S$, then $\beta|_R = \alpha$. We will refer to this as an inclusion of nc dynamical systems and write $(R,G,\alpha) \subseteq (S,G,\beta)$.

An extension $\iota : (R,G,\alpha) \to (S,G,\beta)$ is {\em essential} if whenever $\phi : (S,G,\beta) \to (T,G,\gamma)$ is a morphism such that $\phi \circ \iota$ is an embedding, then $\phi$ is an embedding.

For an nc dynamical system $(R,G,\alpha)$, consider the injective envelope $I(R)$ of $R$ in the category of unital operator systems from Section \ref{subsec:injective-envelopes}. Recall that $R \subseteq I(R)$. The action $\alpha$ extends uniquely to an action of $G$ on the injective envelope $I(R)$ that we continue to denote by $\alpha$, yielding an inclusion of nc dynamical systems $(R,G,\alpha) \subseteq (I(R),G,\alpha)$. 

An nc dynamical system $(R,G,\alpha)$ is {\em injective} if whenever $\iota : (S,G,\beta) \to (T,G,\gamma)$ is an embedding and $\phi : (S,G,\beta) \to (R,G,\alpha)$ is a morphism, then there is a morphism $\psi : (T,G,\gamma) \to (R,G,\alpha)$ such that $\psi \circ \iota = \phi$. 

Hamana \cite{Hamana1985} proved the existence of injective envelopes in the category of nc dynamical systems. Specifically, for an nc dynamical system $(R,G,\alpha)$, there is a unique injective essential extension $((I_G(R),G,I_G(\alpha)),\iota)$.

The operator system $I_G(R)$ is also injective in the category of unital operator systems, meaning in particular that it is a C*-algebra, although it is generally larger than $I(R)$. The injectivity of $(I_G(R),G,I_G(\alpha))$ and the essentiality of the inclusion $(R,G,\alpha) \subseteq (I(R),G,\alpha)$ implies that $\iota$ extends to an embedding of $(I(R),G,\alpha)$ into $(I_G(R),G,I_G(\alpha))$. In fact, this extension is a *-homomorphism. We will identify $I(R)$ with its image in $I_G(R)$, and we will write $\alpha$ for the action of $G$ on $I_G(R)$. This yields the following essential inclusions of nc dynamical systems
\[
(R,G,\alpha) \subseteq (I(R),G,\alpha) \subseteq (I_G(R),G,\alpha). 
\]

For a non-unital C*-dynamical system $(A,G,\alpha)$, the action $\alpha$ extends uniquely to an action of $G$ on the minimal unitization $\tilde{A}$ of $A$ that we continue to denote by $\alpha$. The inclusion
\[
(A,G,\alpha) \subseteq (\tilde{A},G,\alpha)
\]
is {\em essential} in the sense that whenever $\phi : (\tilde{A},G,\alpha) \to (R,G,\alpha)$ is a morphism in the category of nc dynamical systems such that $\phi|_A$ is an embedding, then $\phi$ is an embedding. We will write $I_G(A) := I_G(\tilde{A})$ and define the injective envelope of $(A,G,\alpha)$ to be $(I_G(\tilde{A}),G,\alpha)$. Equivalently, the injective envelope of $(A,G,\alpha)$ is the injective envelope of $(\tilde{A},G,\alpha)$. This yields essential inclusions
\[
(A,G,\alpha) \subseteq (\tilde{A},G,\alpha) \subseteq (I(A),G,\alpha) \subseteq (I_G(A),G,\alpha).
\]

For a C*-dynamical system $(A,G,\alpha)$, Hamana \cite{Hamana1985}*{Theorem 3.4} established the following relationship between injective envelopes and reduced crossed products:
\[
A \times_\lambda G \subseteq I(A) \times_\lambda G \subseteq I_G(A) \rtimes_\lambda G \subseteq I(A \times_\lambda G),
\]
and moreover each of these inclusions is essential. In particular,
\[
I(A \times_\lambda G) = I(I(A) \times_\lambda G) = I(I_G(A) \times_\lambda G).
\]

The action $\alpha$ restricts to an action of $G$ on the center $Z(A)$ that we continue to denote by $\alpha$, yielding a C*-dynamical system $(Z(A),G,\alpha)$. By Lemma 6.2 and Lemma 6.3 of \cite{Hamana1985},
\[
(Z(A),G,\alpha) \subseteq (Z(I(A)),G,\alpha) \subseteq (Z(I_G(A)),G,\alpha)
\]
and
\[
Z(I(A))^G = Z(I_G(A))^G.
\]

We will also require the definition of a normal map between C*-algebras. We will typically apply this definition to inclusions. 

\begin{defn}
Let $A$ and $B$ be C*-algebras. A positive linear map $\phi : A \to B$ is {\em normal} if whenever $a = \sup_A (a_i)$ for a bounded upward directed net $(a_i)$ in $A$, then $\phi(a) = \sup_B (\phi(a_i))$. 
\end{defn}
 
The following result asserts that many of the inclusions we will consider are normal. The proof follows as in the proof of Lemma 3.1 and Corollary 3.5 from \cite{Hamana1985}. 

\begin{lem} \label{lem:sup-preserving}
Let $(A,G,\alpha)$ be a C*-dynamical system. Then the inclusions
\[
A \subseteq I(A) \subseteq I_G(A) \subseteq I(A \times_\lambda G)
\]
are normal.
\end{lem}

\section{Induction and imprimitivity} \label{sec:imprimitivity}

In this section we will introduce a new notion of induction for C*-dynamical systems, along with a corresponding imprimitivity theory that will yield tensor product decompositions with important structural implications for induced systems and their reduced crossed products. However, a key point is that these tensor product decompositions are necessarily at the level of the injective envelope. We will see that a C*-dynamical system and its reduced crossed product may not decompose into tensor products, despite the fact that their injective envelopes do. The results in this section will be of central importance for the remainder of this paper. 

Consider the subcategory of C*-dynamical systems over a fixed group $G$ consisting of C*-dynamical systems $(A,G,\alpha)$ with $A$ injective. It is natural to consider a C*-dynamical system $(B,H,\beta)$ a subsystem of $(A,G,\alpha)$ if $H$ is (isomorphic to) a subgroup of $G$, $B$ is (isomorphic to) an $H$-invariant regular ideal of $A$ and $\beta$ is obtained by restricting $\alpha$. This notion naturally gives rise to restriction functors between appropriate categories. 

There are also natural induction functors. Specifically, let $(B,H,\beta)$ be a C*-dynamical system with $B$ injective. For any group embedding $H \leq G$, there is a natural way to induce a C*-dynamical system $(A,G,\alpha)$ from $(B,H,\beta)$ such that $A$ is injective, $B$ is (isomorphic to) a regular ideal of $A$ and $\alpha$ extends $\beta$. Namely, let $A := \ell^\infty(G/H) \overline{\otimes} B$, where $\overline{\otimes}$ denotes the monotone complete tensor product from Section \ref{sec:preliminaries}. The following result implies that $A$ is injective.

\begin{lem} \label{lem:easy-inj-env}
Let $X$ be a discrete space and let $B$ be an injective C*-algebra. Then the inclusions
\[
c_0(X) \otimes B \subseteq \ell^\infty(X) \otimes B \subseteq \ell^\infty(X) \overline{\otimes} B
\]
are essential. Hence
\[
I(c_0(X) \otimes B) = I(\ell^\infty(X) \otimes B) = \ell^\infty(X) \overline{\otimes} B.
\]
\end{lem}

\begin{proof}
It follows as in the proof of \cite{Hamana1982a}*{Proposition 3.11} that $I(\ell^\infty(X) \otimes B) = \ell^\infty(X) \overline{\otimes} B$. Also, the fact that $c_0(X)$ is an essential ideal in $\ell^\infty(X)$ implies that $c_0(X) \otimes B$ is an essential ideal in $\ell^\infty(X) \otimes B$. Hence the inclusion $c_0(X) \otimes B \subseteq \ell^\infty(X) \otimes B$ is essential, and it follows that $I(c_0(X) \otimes B) = I(\ell^\infty(X) \otimes B)$.
\end{proof}

Now let $p = \delta_e \otimes 1 \in Z(A)$. Then $p$ is a central projection and $Ap = \delta_e \otimes B$, so $\delta_e \otimes B \cong B$ is a regular ideal in $A$. The action $\alpha$ on $A$ is obtained by letting $T \subseteq G$ be a left transversal for $H$ in $G$ and defining $\alpha : G \to \operatorname{Aut}(c_0(G/H) \otimes B)$ by
\[
\alpha_g(\delta_{tH} \otimes b) = \delta_{sH} \otimes \beta_h(b), \quad g \in G,\ t \in T,\ b \in B,
\]
where $gt = sh$ for $s \in T$ and $h \in H$. Since the inclusion $\ell^\infty(G/H) \otimes B \subseteq A$ is essential by Lemma \ref{lem:easy-inj-env}, it follows that this action has a unique extension to an action on $A$.

\begin{defn} \label{defn:induced}
Let $(A,G,\alpha)$ be a C*-dynamical system.
\begin{enumerate}
\item We will say that a C*-dynamical system $(J,H,\beta)$ is a \emph{regular C*-dynamical subsystem} of $(A,G,\alpha)$ if $H \leq G$ is a subgroup, $J \trianglelefteq A$ is an $H$-invariant regular ideal and $\beta$ is obtained by restricting $\alpha$. 

\item If $A$ is injective, then we will say that $(A,G,\alpha)$ is {\em induced} if there is a regular C*-dynamical subsystem $(J,H,\beta)$ such that $(A,G,\alpha)$ is induced by $(J,H,\beta)$ as above, meaning that the inclusion $J \subseteq A$ extends to a $G$-equivariant isomorphism $A \cong \ell^\infty(G/H) \overline{\otimes} J$. 

\item Let $(J,H,\beta)$ be a regular C*-dynamical subsystem of $(A,G,\alpha)$. We will say that the ideals $\{\alpha_g(J)\}_{g \in G/H}$ form a {\em system of imprimitivity} for $(A,G,\alpha)$ relative to $(J,H,\beta)$ if $\{\alpha_g(J)\}_{g \in G/H}$ are pairwise orthogonal with $\vee_{g \in G/H} \alpha_g(J) = A$ and $H = \{g \in G : \alpha_g(J) = J\}$.
\end{enumerate}
\end{defn}

\begin{rem}
Note that (1) and (3) do not require $A$ to be injective. This fact will be important for Lemma \ref{lem:imprimitivity-detect-induced-injective-envelope}.
\end{rem}

As expected, the existence of a system of imprimitivity for a C*-dynamical system over an injective C*-algebra is equivalent to being induced. 

\begin{prop} \label{prop:imprimitivity}
Let $(A,G,\alpha)$ be a C*-dynamical system with $A$ injective and let $(J,H,\beta)$ be a regular C*-dynamical subsystem. The ideals $\{\alpha_g(J)\}_{g \in G/H}$ form a system of imprimitivity for $(A,G,\alpha)$ relative to $(J,H,\beta)$ if and only if $(A,G,\alpha)$ is induced from $(J,H,\beta)$, in which case
\[
A \cong \ell^\infty(G/H) \overline{\otimes} J.
\]
\end{prop}

\begin{proof}
First, it is clear that if $(A,G,\alpha)$ is induced from $(J,H,\beta)$, then the ideals $\{\alpha_g(J)\}_{g \in G/H}$ are pairwise orthogonal with $\vee_{g \in G/H} \alpha_g(J) = A$ and $H = \{g \in G : \alpha_g(J) = J\}$. Hence $\{\alpha_g(J)\}_{g \in G/H}$ is a system of imprimitivity for $(A,G,\alpha)$ relative to $(J,H,\beta)$.

Conversely, suppose that $\{\alpha_g(J)\}_{g \in G/H}$ is a system of imprimitivity for $(A,G,\alpha)$ relative to $(J,H,\beta)$. Let $p \in Z(A)$ be the central projection such that $J = Ap$. Let $T \subseteq G$ be a left transversal of $H$ in $G$. Then by hypothesis the family $\{\alpha_g(p)\}_{g \in T}$ consists of mutually orthogonal central projections in $Z(A)$ that partition the identity. Viewing elements of $\ell^\infty(G/H) \otimes J$ as uniformly bounded functions from $G/H$ to $J$, there is a natural equivariant embedding $\iota$ from $\ell^\infty(G/H) \otimes J$ into $A$ defined by
\[
\iota(f) = \sum_{t \in T} \alpha_t(f(tH)) \alpha_t(p), \quad f\in \ell^\infty(G/H) \otimes J,
\]
where the sum on the right hand side denotes the order limit of the family of sums taken over finite subsets of $T$. The fact that $\iota$ is a well-defined embedding follows from the orthogonality of the projections $\{\alpha_t(p)\}_{t \in T}$ and \cite{Hamana1982a}*{Lemma 1.5}. We identify $\ell^\infty(G/H) \otimes J$ with its image under $\iota$.

We claim that the inclusion $\ell^\infty(G/H) \otimes J \subseteq A$ is essential. To prove this claim, it suffices  to prove that if $\phi : A \to A$ is unital and completely positive and $\phi|_{\ell^\infty(G/H) \otimes J} = \id_A|_{\ell^\infty(G/H) \otimes J}$, then $\phi = \id_A$. For this, since the projections $\{\alpha_t(p)\}_{t \in T}$ are central and orthogonal with $\vee_{t \in T} \alpha_t(p) = 1$, it suffices to prove that for $a \in A$ and $t \in T$, $\phi(a)\alpha_t(p) = a \alpha_t(p)$. But $\phi(\alpha_t(p)) = \alpha_t(p)$ by assumption, so $\alpha_t(p)$ belongs to the multiplicative domain of $\phi$. Hence for $a \in A$ and $t \in T$, $\phi(a)\alpha_t(p) = \phi(a \alpha_t(p)) = a \alpha_t(p)$. Therefore, $\phi = \id_A$ as claimed. Since the inclusion $\ell^\infty(G/H) \otimes J \subseteq I(A)$ is essential, it now follows from Lemma \ref{lem:easy-inj-env} that
\[
A = I(A) = I(\ell^\infty(G/H) \otimes J) = \ell^\infty(G/H) \overline{\otimes} J. \qedhere
\]
\end{proof}

The definition of an induced C*-dynamical system from Definition \ref{defn:induced} only applies to C*-dynamical systems over injective C*-algebras. However, the definition of a regular C*-dynamical subsystem and the definition of a system of imprimitivity relative to a regular C*-dynamical subsystem do not require injectivity. The next result implies that for a C*-dynamical system $(A,G,\alpha)$, the existence of a system of imprimitivity relative to a regular C*-dynamical subsystem implies that the C*-dynamical system $(I(A),G,\alpha)$ is induced. Conversely, if $(I(A),G,\alpha)$ is induced, then this can be detected from the existence of a system of imprimitivity relative to a regular C*-dynamical subsystem for $(A,G,\alpha)$. 

\begin{lem} \label{lem:imprimitivity-detect-induced-injective-envelope}
Let $(A,G,\alpha)$ be a C*-dynamical system and let $(J,H,\beta)$ be a regular C*-dynamical subsystem. If $\{\alpha_g(J)\}_{g \in G/H}$ is a system of imprimitivity for $(A,G,\alpha)$ relative to $(J,H,\beta)$, then the C*-dynamical system $(I(A),G,\alpha)$ is induced from the regular C*-dynamical subsystem $(I(J),H,\beta)$. Conversely, if $(I(A),G,\alpha)$ is induced from a regular C*-dynamical subsystem $(K,H,\beta)$, then letting $J = A \cap K$, $J$ is an $H$-invariant regular ideal in $A$ and $\{\alpha_g(J)\}_{g \in G/H}$ is a system of imprimitivity for $(A,G,\alpha)$ relative to the regular C*-dynamical subsystem $(J,H,\beta)$. 
\end{lem}

\begin{proof}
Let $(J,H,\beta)$ be a regular C*-dynamical subsystem of $(A,G,\alpha)$ such that $\{\alpha_g(J)\}_{g \in G/H}$ forms a system of imprimitivity for $(A,G,\alpha)$ relative to $(J,H,\beta)$. Then $\{\alpha_g(I(J))\}_{g \in G/H}$ forms a system of imprimitivity for $(I(A),G,\alpha)$ relative to $(I(J),H,\beta)$. Hence by Proposition \ref{prop:imprimitivity}, $(I(A),G,\alpha)$ is induced from $(I(J),H,\beta)$. 

Conversely, if $(I(A),G,\alpha)$ is induced from a regular C*-dynamical subsystem $(K,H,\beta)$, then by Proposition \ref{prop:imprimitivity}, $\{\alpha_g(K)\}_{g \in G/H}$ is a system of imprimitivity for $(I(A),G,\alpha)$ relative to $(K,H,\beta)$. Let $J = A \cap K$. Then $J$ is an $H$-invariant regular ideal of $A$,  $(J,H,\beta)$ is a regular C*-dynamical subsystem of $(A,G,\alpha)$ and $\{\alpha_g(J)\}_{g \in G/H}$ is a system of imprimitivity for $(A,G,\alpha)$ relative to $(J,H,\beta)$.
\end{proof}

Lemma \ref{lem:imprimitivity-detect-induced-injective-envelope} justifies the following extension of the notion of an induced C*-dynamical system from Definition \ref{defn:induced} to C*-dynamical systems where the underlying C*-algebra is not necessarily injective.

\begin{defn} \label{defn:induced-general} We will say that a C*-dynamical system $(A,G,\alpha)$ is \emph{induced} from a regular C*-dynamical subsystem $(J,H,\beta)$ if the ideals $\{\alpha_g(J)\}_{g \in G/H}$ form a system of imprimitivity for $(A,G,\alpha)$ relative to $(J,H,\beta)$. In this case we will say that $(A,G,\alpha)$ is {\em induced}.
\end{defn}

The next result is an analogue of Green's imprimitivity theorem \cite{Green1980}*{Theorem 2.13}.

\begin{thm} \label{thm:imprimitivity-reduced-crossed-product}
Let $(A,G,\alpha)$ be a C*-dynamical system. Suppose that $(A,G,\alpha)$ is induced from a regular C*-dynamical subsystem $(J,H,\beta)$. Then
\[
I(A) \cong \ell^\infty(G/H) \overline{\otimes} I(J)
\]
and
\[
I(A \times_\lambda G) \cong \mathcal{B}(\ell^2(G/H)) \overline{\otimes} I(J \times_\lambda H).
\]
\end{thm}

\begin{proof}
Since $I(A \times_\lambda G) = I(I(A) \times_\lambda G)$, we can assume that $A$ is injective. Then by Proposition \ref{prop:imprimitivity}, $A \cong \ell^\infty(G/H) \overline{\otimes} J$. Identify these C*-algebras. By Lemma \ref{lem:easy-inj-env}, the inclusion
\[
c_0(G/H) \otimes J \subseteq \ell^\infty(G/H) \otimes J
\]
is essential, and clearly $c_0(G/H) \otimes J$ is $G$-invariant. Hence by \cite{Hamana1985}*{Theorem~3.4},
\[
I((c_0(G/H) \otimes J) \times_\lambda G) = I((\ell^\infty(G/H) \otimes J) \times_\lambda G).
\]
Now by Green's imprimitivity theorem \cite{Green1980}*{Theorem 2.13},
\[
(c_0(G/H) \otimes J) \times_\lambda G \cong \mathcal{K}(\ell^2(G/H)) \otimes (J \times_\lambda H).
\]
Also, the inclusion
\[
\mathcal{K}(\ell^2(G/H)) \otimes (J \times_\lambda H) \subseteq \mathcal{B}(\ell^2(G/H)) \otimes (J \times_\lambda H)
\]
is essential. Hence taking injective envelopes on both sides and applying \cite{Hamana1982a}*{Proposition~3.11} gives
\begin{align*}
I(A \times_\lambda G) &= I(\ell^\infty(G/H) \otimes J) \times_\lambda G) = I((c_0(G/H) \otimes J) \times_\lambda G) \\
&\cong I(\mathcal{K}(\ell^2(G/H)) \otimes (J \times_\lambda H)) = \mathcal{B}(\ell^2(G/H)) \overline{\otimes} I(J \times_\lambda H). \qedhere
\end{align*}

\end{proof}

In practice, we will need to work with C*-dynamical systems $(A,G,\alpha)$ that are not themselves induced in a non-trivial way, but that instead have a $G$-invariant regular ideal $K \trianglelefteq A$ such that the C*-dynamical system $(K,G,\alpha)$ is induced by a regular C*-dynamical subsystem.

Recall that if $A$ is injective and $K \trianglelefteq A$ is a regular ideal, then there is a central projection $p \in Z(A)$ such that $K = Ap$. In particular, $K$ is a summand of $A$. However, if $A$ is not injective, then it is not necessarily true that a regular ideal in $A$ is a summand of $A$. So in general, it is necessary to refer to regular ideals instead of summands. This motivates the following definition.

\begin{defn} \label{defn:sub-induced} We will say that a C*-dynamical system $(A,G,\alpha)$ is {\em sub-induced} from a regular C*-dynamical subsystem if there is a nonzero $G$-invariant regular ideal $K \trianglelefteq A$ such that the corresponding C*-dynamical subsystem $(K,G,\alpha)$ is induced from a regular C*-dynamical subsystem $(J,H,\beta)$. In this case, we will say that $(A,G,\alpha)$ is sub-induced from the regular C*-dynamical subsystem $(J,H,\beta)$.
\end{defn}

\begin{rem}
Let $(A,G,\alpha)$ be a C*-dynamical system, let $K \trianglelefteq A$ be a $G$-invariant regular ideal and let $(J,H,\beta)$ be a regular C*-dynamical subsystem of the C*-dynamical system $(K,G,\alpha)$. Then $J$ is also a regular ideal in $A$, so that $(J,H,\beta)$ is also a regular C*-dynamical subsystem of $(A,G,\alpha)$. Hence there is no ambiguity in Definition \ref{defn:sub-induced}. 
\end{rem}

\begin{prop} \label{prop:sub-induced}
A C*-dynamical subsystem $(A,G,\alpha)$ is sub-induced from a regular C*-dynamical subsystem $(J,H,\beta)$ if and only if $H = \{g \in G : \alpha_g(J) = J\}$ and the ideals $\{\alpha_g(J)\}_{g \in G/H}$ are pairwise orthogonal. 
\end{prop}

\begin{proof}
Let $K = \vee_{g \in G/H} \alpha_g(J)$. Then $K$ is a $G$-invariant regular ideal of $A$ and the ideals $\{\alpha_g(J)\}_{g \in G/H}$ form a system of imprimitivity for $(K,G,\alpha)$ relative to $(J,H,\beta)$, so $(K,G,\alpha)$ is induced from $(J,H,\beta)$. 
\end{proof}

\begin{rem} \label{rem:abelian-sub-induced}
Let $(\rC_0(X),G,\alpha)$ be an abelian C*-dynamical system. A regular C*-dynamical subsystem is of the form $(\rC_0(U),H,\beta)$, where $U \subseteq X$ is a non-empty $H$-invariant regular open subset. By Proposition \ref{prop:sub-induced}, $(\rC_0(X),G,\alpha)$ is sub-induced from $(\rC_0(U),H,\beta)$ if and only if $H = \{h \in G : hU = U\}$ and the ideals $\{\alpha_g(\rC_0(U))\}_{g \in G/H}$ are pairwise orthogonal. The latter condition is equivalent to the subsets $\{gU\}_{g \in G/H}$ being pairwise disjoint.

Hence $(\rC_0(X),G,\alpha)$ is sub-induced if and only if there is a non-empty regular open subset $U \subseteq X$ and a subgroup $H \leq G$ with $H = \{g \in G : gU = U\}$ such that the subsets $\{gU\}_{g \in G/H}$ are pairwise disjoint. In this case, $V = \vee_{g \in G/H} gU \subseteq X$ is a non-empty $G$-invariant regular open subset and the C*-dynamical subsystem $(\rC_0(V),G,\alpha)$ is induced from the regular C*-dynamical subsystem $(\rC_0(U),H,\beta)$.
\end{rem}

The next result is an easy consequence of the above results.

\begin{cor} \label{cor:detect-induced-quotient}
Let $(A,G,\alpha)$ be a C*-dynamical system. If $(A,G,\alpha)$ is sub-induced from a regular C*-dynamical subsystem $(J,H,\beta)$, then $(I(A),G,\alpha)$ is sub-induced from the regular C*-dynamical subsystem $(I(J),H,\beta)$. Conversely, if $(I(A),G,\alpha)$ is sub-induced from a regular C*-dynamical subsystem $(J,H,\beta)$, then $(A,G,\alpha)$ is sub-induced from the regular C*-dynamical subsystem $(A \cap J, H, \beta)$.
\end{cor}

\section{Meandering projections} \label{sec:projections}

In this section, we will introduce the notion of a meandering projection for a C*-dynamical system. This will provide a useful technical device for detecting when a given C*-dynamical system is sub-induced by a regular C*-dynamical subsystem in the sense of Section \ref{sec:imprimitivity}, with the additional requirement that the restriction of some non-trivial automorphism of the subsystem is implemented by a unitary in the injective envelope satisfying a strong invariance property. This condition will play an important role throughout this paper. However, the precise relationship between the meandering projection and the subsystem is potentially quite complicated, and it will take a significant amount of work to establish this result.

The following definition arose naturally from our analysis of relative commutants in injective envelopes of reduced crossed products (see Section~\ref{sec:regular-ideal-intersection-property}), and from our analysis of pseudo-expectations (see Section~\ref{sec:pseudo-expectations}).

\begin{defn} \label{defn:meandering-projection}
Let $(A,G,\alpha)$ be a C*-dynamical system. We will say that a nonzero central projection $p \in Z(A)$ is {\em meandering} for $(A,G,\alpha)$ if there is a group element $r \in G \setminus \{e\}$ and a unitary $u \in Ap$ such that 
\begin{enumerate}
\item $\alpha_r|_{Ap} = \operatorname{Ad}(u)$ and $u$ is $C_G(r)$-invariant;
\item $\sup_F \| \sum_{t \in F} \alpha_t(p) \| < \infty$, where the supremum is taken over all finite subsets $F \subseteq T$ for any left transversal $T$ of $C_G(r)$ in $G$.
\end{enumerate}
\end{defn}

\begin{rem}
Note that since $u^*u = p = uu^*$, the $C_G(r)$-invariance of $u$ implies the $C_G(r)$-invariance of $p$.
\end{rem}

We will require the following technical result.

\begin{lem} \label{lem:collection-of-functions}
Let $X$ be a compact extremally disconnected Hausdorff space and let $\{f_i\}_{i \in I}$ be a collection of functions in $C(X)$. Suppose there is $\lambda > 0$ such that $f_i(x) \geq \lambda$ for all $i \in I$ and $x \in \operatorname{supp}(f_i)$, and that
\[
\sup_F \|\sum_{i \in F} f_i\| < \infty,
\]
where the supremum is taken over all finite subsets $F \subseteq I$. Then
\begin{enumerate}
\item the bounded upward directed net $( \sum_{i \in F} f_i )_F$ has a supremum $f \in C(X)$,
\item the set $\{i \in I : f_i(x) \ne 0\}$ is finite for each $x \in X$,
\item for each finite subset $F \subseteq I$, the set $\{ x \in X : f(x) = \sum_{i \in F} f_i \}$ is clopen, and 
\item the set $\{x \in X : f(x) = \sum_{i \in I} f_i(x) \}$ is open and comeager, where $\sum_{i \in I} f_i(x)$ denotes the pointwise supremum of the net $(\sum_{i \in F} f_i)_F$.
\end{enumerate} 
\end{lem}

\begin{proof}
The net $(\sum_{i \in F} f_i)_F$ is bounded and upward directed by assumption. Since $X$ is extremally disconnected, $C(X)$ is injective by \cite{Gleason1958}*{Theorem 2.5}. Hence $f := \sup_F \sum_{i \in F} f_i \in C(X)$ exists. The norm boundedness of the net $(\sum_{i \in F} f_i)_F$ also clearly implies that $\{i \in I : f_i(x) \ne 0\}$ is finite for each $x \in X$. 

For a finite subset $F \subseteq I$, it is clear that the set $\{ x \in X : f(x) = \sum_{i \in F} f_i \}$ is closed. To see that it is open, observe that its complement is
\[
\{x \in X : f(x) - \sum_{i \in F} f_i(x) \geq \lambda/2 \},
\]
which is clearly closed.

Finally, the function $\sum_{i \in I} f_i$ is bounded and lower semicontinuous, and in particular Borel. Hence it belongs to the C*-algebra $\mathrm{B}(X)$ of bounded Borel functions on $X$. Since $\rC(X)$ is injective, it follows from \cite{Gonshor1970}*{Theorem 1} that $C(X) \cong \mathrm{B}(X)/M$, where $M \trianglelefteq \mathrm{B}(X)$ denotes the ideal of functions in $\mathrm{B}(X)$ supported on a meager set. Further, the composition of the quotient map $q : \mathrm{B}(X) \to \rC(X)$ with the inclusion $\rC(X) \subseteq \mathrm{B}(X)$ is the identity on $\rC(X)$.

Identify $C(X)$ with its image in $\mathrm{B}(X)$. Then $\sum_{i \in I} f_i$ is the supremum of the net $(\sum_{i \in F} f_i)_F$ in $\mathrm{B}(X)$. Hence for finite $F \subseteq I$,
\[
\sum_{i \in F} f_i \leq \sum_{i \in I} f_i \leq f
\]
in $\mathrm{B}(X)$. Applying the quotient map $q$ then gives
\[
\sum_{i \in F} f_i = q\left(\sum_{i \in F} f_i\right) \leq q\left(\sum_{i \in I} f_i\right) \leq q(f) = f
\]
in $\rC(X)$. In particular, $q(\sum_{i \in I} f_i)$ is an upper bound for the net $(\sum_{i \in F} f_i)_F$ such that $q(\sum_{i \in I} f_i) \leq f$. Since $f$ is the least upper bound in $\rC(X)$, it follows that $q(\sum_{i \in I} f_i) = f$. Hence $f$ agrees with $\sum_{i \in I} f_i$ on a comeager set. From above, this set is $\cup_F \{x \in X : f(x) = \sum_{i \in F} f_i(x)\}$, where the union is taken over all finite subsets $F \subseteq I$. From above, this is a union of clopen sets, and hence is open.
\end{proof}

\begin{lem} \label{lem:projection-to-induced}
Let $(A,G,\alpha)$ be a C*-dynamical system. Suppose there is a meandering projection for $(I(A),G,\alpha)$. Then $(I(A),G,\alpha)$ is sub-induced from a regular C*-dynamical subsystem $(J,H,\beta)$ with $r \in FC(H) \setminus \{e\}$ such that $\beta_r = \operatorname{Ad}(u_r)$ for a $C_H(r)$-invariant unitary $u_r \in J$.
\end{lem}

\begin{proof}
Let $p \in Z(I(A))$ be a meandering projection with corresponding group element $r \in G \setminus \{e\}$ and unitary $u \in I(A)p$. Let $T \subseteq G$ be a left transversal of the centralizer $C_G(r)$ in $G$ with $e \in T$. Let $X := \operatorname{Glimm}(Z(I(A)))$, so that $Z(I(A)) \cong \rC(X)$. Note that since $Z(I(A))$ is injective, $X$ is extremally disconnected. Since $p$ is meandering, the collection of functions $\{\alpha_t(p)\}_{t \in T}$ satisfies the hypotheses of Lemma \ref{lem:collection-of-functions}. Let
\[
f = \textstyle{\sup_{\rC(X)}} \{\sum_{t \in F} \alpha_t(p) : F \subseteq T \text{ finite} \}.
\]
Note that $f$ does not depend on $T$ because $p$ is $C_G(r)$-invariant. Since $p \ne 0$, it follows that there is a finite subset $F \subseteq T$ with $e \in F$ such that the set
\begin{align*}
U &:= \left\{x \in X : f(x) = \sum_{t \in F} \alpha_t(p)(x) \right\} \bigcap \left(\bigcap_{t \in F} \operatorname{supp}(\alpha_t(p))\right) \\
&\phantom{:}\subseteq \left(\bigcap_{t \in F} \operatorname{supp}(\alpha_t(p))\right) \setminus \left( \bigcup_{t \in T \setminus F} \operatorname{supp}(\alpha_t(p)) \right)
\end{align*}
is clopen and non-empty. Note that $U$ is contained in the support of $p$.

Define the subgroup $H \leq G$ by
\[
H := \{h \in G : hU = U \},
\]
and note that $r \in H$ because $r$ acts trivially on the support of $p$. It follows from the definition of $U$ and the $C_G(r)$-invariance of $p$ that we can also write
\[
H = \{h \in G : h F \cdot C_G(r) = F \cdot C_G(r) \} = \{h \in G : hU \cap U \ne \emptyset \}.
\]
Let $q = \chi_U \in Z(I(A))$ and note that $q \leq p$. Define the regular ideal $J$ in $I(A)$ by $J := I(A)q$. It follows immediately from above that
\[
H = \{h \in G : \alpha_h(J) = J \} = \{h \in G : \alpha_h(J) \cap J \ne 0 \}.
\]
By the $H$-invariance of $J$, we obtain a regular C*-dynamical subsystem $(J,H,\beta)$ such that the ideals $\{\alpha_g(J)\}_{g \in G/H}$ are pairwise orthogonal. Hence by Proposition \ref{prop:sub-induced}, $(A,G,\alpha)$ is sub-induced from $(J,H,\beta)$.

It remains to prove that $r \in FC(H)$ and $\beta_r = \operatorname{Ad}(u_r)$ for a $C_H(r)$-invariant unitary $u_r \in J$. For this, first note that since $q \leq p$ and $r$ acts trivially on the support of $p$, $q$ is also $r$-invariant. In particular, $\alpha_r(J) = J$. Next, choose $h \in H$. Then since $r \in F$, $hr \in F \cdot C_G(r)$, so $h = t s r^{-1}$ for some $t \in F$ and $s \in C_G(r)$, giving $hrh^{-1} = trt^{-1}$. Hence the $H$-conjugacy class $\operatorname{Cl}_H(r)$ of $r$ satisfies
\[
\operatorname{Cl}_H(r) = \{t r t^{-1} : t \in F\}.
\]
In particular, since $F$ is finite, $r \in FC(H)$. Finally, since $u$ is $C_G(r)$-invariant and $q$ is $H$-invariant, the unitary $u_r := uq \in J$ is $C_H(r)$-invariant and satisfies $\beta_r = \operatorname{Ad}(u_r)$. 
\end{proof}

\begin{lem} \label{lem:induced-to-relative-commutant}
Let $(A,G,\alpha)$ be a C*-dynamical system. Suppose that $(I(A),G,\allowbreak\alpha)$ is sub-induced from a regular C*-dynamical subsystem $(J,H,\beta)$ with $r \in FC(H) \setminus \{e\}$ such that $\beta_r = \operatorname{Ad}(u_r)$ for a $C_H(r)$-invariant unitary $u_r \in J$. Then $Z(I(A))^G \ne Z(I(A \times_\lambda G))$.
\end{lem}

\begin{proof}
Let $T \subseteq H$ be a left transversal of the centralizer $C_H(r)$ in $H$. Note that $T$ is finite because $r \in FC(H)$. Define $b \in I(A) \times_\lambda H$ by
\[
b = \sum_{t \in T} \alpha_t(u_r^* \lambda_r).
\]
Then as in the proof of \cite{GeffenUrsu2023}*{Theorem 4.7}, $b \notin I(A)$, but $b \in Z(J \times_\lambda H) \subseteq Z(I(J \times_\lambda H))$. 

By Theorem \ref{thm:imprimitivity-reduced-crossed-product}, we can identify $I(A \times_\lambda G)$ with $I(J \times_\lambda H) \overline{\otimes} \mathcal{B}(\ell^2(G/H))$. Then from above, $b \otimes 1 \notin I(A)$, but $b \otimes 1 \in Z(I(J \times_\lambda H) \overline{\otimes} \mathcal{B}(\ell^2(G/H)))$. In particular, $Z(I(A))^G \ne Z(I(A \times_\lambda G))$.
\end{proof}

The next result is inspired by \cite{GeffenUrsu2023}*{Lemma 4.5}.

\begin{lem} \label{lem:relative-commutant-reduction}
Let $A \subseteq B \subseteq I(B)$ and $B \subseteq C$ be inclusions of C*-algebras with $C$ monotone complete. If $Z(I(B)) \setminus A \ne \emptyset$, then $(B' \cap C) \setminus A \ne \emptyset$. 
\end{lem}

\begin{proof}
The center $Z(I(B))$ is an abelian monotone complete C*-algebra, so the span of the projections is dense in $Z(I(B))$. Hence there is a nonzero projection $q \in Z(I(B)) \setminus A$. The set
\[
X = \{b \in B : 0 \leq b \leq q \text{ and } \|b\| < 1\}
\]
is non-empty, bounded and upward directed (see e.g. \cite{SaitoWright2015}*{Proposition 2.1.15}) and $q = \sup_{I(B)} X$.

Define $p \in C$ by $p = \sup_C X$. Then $p$ is a projection, but we will not require that fact here. Since $X$ is invariant under conjugation by unitaries in $B$, $p \in B' \cap C$. Also, $p \notin A$, since otherwise if $p \in A$, then since $A \subseteq B$ and the inclusion $B \subseteq I(B)$ is normal by Lemma \ref{lem:sup-preserving}, it would follow that $p = \sup_{I(B)} X = q$, contradicting the fact that $q \notin A$. 
\end{proof}

\begin{lem} \label{lem:relative-commutant-to-projection}
Let $(A,G,\alpha)$ be a C*-dynamical system. If $Z(I(A))^G \ne Z(I(A \times_\lambda G))$, then $(I(A),G,\alpha)$ has a meandering projection.
\end{lem}

\begin{proof}
For this proof, we will require the monotone complete crossed product $M(I(A),G)$ (see \cite{Hamana1985}). Applying Lemma \ref{lem:relative-commutant-reduction} to the inclusions $I(A) \subseteq I(A) \times_\lambda G$ and $I(A) \times_\lambda G \subseteq M(I(A),G)$, we obtain self-adjoint $b \in ((I(A) \times_\lambda G)' \cap M(I(A),G)) \setminus I(A)$. Let $b \sim \sum_{g \in G} a_g \lambda_g$ be the Fourier series for $b$, so that $a_g := E(b \lambda_g^*) \in I(A)$ for $g \in G$.

The fact that $b$ commutes with $A \times_\lambda G$ implies that for $t \in G$, $\alpha_t(a_g) = a_{tgt^{-1}}$, and that for $a \in I(A)$, $a_g a = \alpha_{g^{-1}}(a) a_g$ (see \cite{Hamana1985}*{Lemma 10.3}). In particular, $a_g^* a_g a = a a_g^* a_g$, implying $a_g^* a_g \in Z(I(A))$ and hence that $|a_g| \in Z(I(A))$.

Let $u_{g^{-1}} \in I(A)$ be the partial isometry obtained from the polar decomposition of $a_g$, so that $a_g = u_{g^{-1}} |a_g|$, and let $p_g \in Z(I(A))$ be the support projection for $|a_g|$. Then it follows from above that for $t \in G$, $\alpha_t(u_g) = u_{tgt^{-1}}$, and for $a \in I(A)$, $u_g a = \alpha_g(a) u_g$. Furthermore, $\alpha_g|_{I(A)p_g} = \operatorname{Ad}(u_g)$, so in particular $\alpha_g$ acts trivially on $Z(I(A))p_g$. 

Since $b \notin I(A)$, there is $r \in G \setminus \{e\}$ such that $a_r \ne 0$. Choose $\lambda > 0$ such that $2\lambda$ is in the range of $|a_r|$, and let $U_r \subseteq \operatorname{Glimm}(I(A))$ denote the closure of the open set
\[
\{x \in \operatorname{Glimm}(I(A)) : |a_r|(x) > \lambda \}. 
\]
Then $U_r$ is clopen and $C_G(r)$-invariant. Hence letting $p = \chi_{U_r} \in Z(I(A))$ denote the corresponding central projection, $p$ is also $C_G(r)$-invariant.

We claim that $p$ is meandering for $(I(A),G,\alpha)$. To see this, let $u = u_r p$. Then $p \leq p_r$, so $u$ is $C_G(r)$-invariant and $\alpha_r|_{I(A)p} = \operatorname{Ad}(u)$. Let $T \subseteq G$ be a left transversal of $C_G(r)$ in $G$. Then for a finite subset $F \subseteq T$,
\[
\left\| \sum_{t \in F} \alpha_t(p) \right\| \leq \lambda^{-2} \left\| \sum_{t \in F} a_{trt^{-1}}^* a_{trt^{-1}} \right\| \leq \lambda^{-2} \|E(b^2)\|.
\]
Since $p$ is $C_G(r)$-invariant, this upper bound does not depend on $T$ nor $F$. It follows that $p$ is meandering for $(I(A),G,\alpha)$.
\end{proof}

\begin{prop} \label{prop:projection-equivalences}
Let $(A,G,\alpha)$ be a C*-dynamical system. The following are equivalent:
\begin{enumerate}
\item $(I(A),G,\alpha)$ has a meandering projection;
\item $(I_G(A),G,\alpha)$ has a meandering projection;
\item $(I(A),G,\alpha)$ is sub-induced from a regular C*-dynamical subsystem $(J,H,\beta)$ with $r \in FC(H) \setminus \{e\}$ such that $\beta_r = \operatorname{Ad}(u_r)$ for a $C_H(r)$-invariant unitary $u_r \in J$;
\item $Z(I(A))^G \ne Z(I(A \times_\lambda G))$.
\end{enumerate}
\end{prop}

\begin{proof}
The implication (1) implies (3) is Lemma \ref{lem:projection-to-induced}, The implication (3) implies (4) is Lemma \ref{lem:induced-to-relative-commutant}, and the implication (4) implies (1) is Lemma \ref{lem:relative-commutant-to-projection}.

For the implication (2) implies (4), suppose that $(I_G(A),G,\alpha)$ has a meandering projection. Then since $I_G(A)$ is injective, Lemma \ref{lem:projection-to-induced} implies that $(I_G(A),G,\alpha)$ is sub-induced from a regular C*-dynamical subsystem $(J,H,\beta)$ with $r \in FC(H) \setminus \{e\}$ such that $\beta_r = \operatorname{Ad}(u_r)$ for a $C_H(r)$-invariant unitary $u_r \in J$. Applying Lemma \ref{lem:induced-to-relative-commutant} yields
\[
Z(I(A))^G = Z(I_G(A))^G \ne Z(I(I_G(A) \times_\lambda G)) = Z(I(A \times_\lambda G)).
\]
Finally, the implication (4) implies (2) follows from the equality $Z(I(A))^G = Z(I_G(A))^G$ and another application of \ref{lem:relative-commutant-to-projection}. 
\end{proof}

\section{Derivations and approximating unitaries} \label{sec:intrinsic}

In this section we will establish two characterizations of a C*-dynamical system $(A,G,\alpha)$ with the property that the C*-dynamical system $(I(A),G,\alpha)$ is sub-induced by a regular C*-dynamical subsystem $(J,H,\beta)$ with $r \in FC(H) \setminus \{e\}$ such that $\beta_r = \operatorname{Ad}(u_r)$ for a $C_H(r)$-invariant unitary $u_r \in J$. We will need to consider this property throughout this paper.

This property was already characterized in Proposition \ref{prop:projection-equivalences} from Section~\ref{sec:projections}. However, the characterization provided there is ``extrinsic,'' in terms of the existence of meandering projections in the injective envelope. The characterizations that we will establish here are ``intrinsic,'' in terms of the C*-dynamical system itself, and do not require any knowledge of the injective envelope. 

This presents two difficulties. First, we require an intrinsic characterization of $(A,G,\alpha)$ with the property that $(I(A),G,\alpha)$ is sub-induced as above. Fortunately, such a characterization is provided by Corollary \ref{cor:detect-induced-quotient} from Section \ref{sec:imprimitivity}. Second, assuming that $(A,G,\alpha)$ is sub-induced by $(J,H,\beta)$ as above, we require an intrinsic characterization of the existence of $r \in FC(H) \setminus \{e\}$ such that $\beta_r = \operatorname{Ad}(u_r)$ for a $C_H(r)$-invariant unitary $u_r \in J$. The main difficulty here is the $C_H(r)$-invariance of the unitary implementing the automorphism $\beta_r$ above. Without this condition, we would essentially be seeking an intrinsic characterization of the existence of an automorphism that extends to an inner automorphism on a piece of the injective envelope, and such characterizations are already known (see e.g. \cite{Hamana1985}*{Theorem 7.3}). However, as we will see in Example \ref{ex:cant-drop-invariance-condition}, it is generally not possible to drop this invariance condition.

For the case when $A$ is separable, Geffen and Ursu \cite{GeffenUrsu2023} established a characterization of this property in terms of distance estimates to inner automorphisms on $A$ implemented by unitaries approximately commuting with the subgroup $C_H(r)$. This can be seen as an equivariant analogue of Elliott's characterization of proper outerness for separable C*-algebras from \cite{Elliott1980}*{Theorem 2.3}. See Lemma \ref{lem:geffen-ursu-unitary} below.

We will establish another characterization of this property in terms of the existence of $C_H(r)$-commuting *-derivations on $A$. This characterization does not require the separability of $A$, and it can be seen as an equivariant counterpart of Olesen and Pedersen's characterization of proper outerness from \cite{OlesenPedersen1982}*{Theorem 4.3}.

\begin{lem} \label{lem:derivation-spectrum}
Let $A$ be an injective C*-algebra and let $u \in A$ be a unitary. Let $H \leq \operatorname{Aut}(A)$ be a subgroup of automorphisms that leave $u$ invariant. For $\epsilon > 0$, there is an $H$-invariant open projection $p \in A \cap \{u\}'$ and a point $\lambda$ in the spectrum of $u$ such that $\|up - \lambda p\| < \epsilon$.
\end{lem}

\begin{proof}
Let $C = \mathrm{C}^*(u)$. Let $X \subseteq \mathbb{T}$ denote the spectrum of $u$ and identify $C$ with $\mathrm{C}(X)$. Choose $\lambda \in X$ and let $U \subseteq X$ be an open subset with $\lambda \in U$ and $\operatorname{diam}(U) < \epsilon$.  Let $q = \chi_U \in C^{**}$ denote the corresponding open projection. Then $\|uq - \lambda q\| < \epsilon$.

We claim that $I(C)$ embeds as a C*-subalgebra of $A$. To see this, first let $B \subseteq A$ be a MASA containing $\mathrm{C}^*(u)$. Then since $A$ is injective, and hence monotone complete, $B$ is monotone complete (cf. \cite{Hamana1981}*{Remark 4.15}). Therefore, since $B$ is abelian, it is injective.

By Lemma \ref{lem:equivalence-abelian-injectivity}, $B$ is also injective in the category of unital abelian C*-algebras with unital *-homomorphisms as morphisms. Hence the inclusion map $C \subseteq B$ extends to a *-homomorphism from $I(C) \to B$. By the essentiality of the inclusion $C \subseteq I(C)$, this *-homomorphism is injective, and we identify $I(C)$ with its image.

Let $\phi : C^{**} \to I(C)$ be a morphism extending the inclusion $C \subseteq I(C)$ and let $p = \phi(q)$. Then by \cite{Hamana1982a}*{Lemma 6.4}, $p$ is an open projection in $I(C)$, meaning that there is an increasing net $\{c_i\}$ in $C$ such that $p = \sup c_i$. Since elements in $C$ are pointwise $H$-invariant, $p$ is $H$-invariant. Furthermore, since $u$ belongs to the multiplicative domain of $\phi$, it follows from above that
\[
\|up - \lambda p\| = \|u\phi(q) - \lambda \phi(q)\| = \|\phi(uq - \lambda q)\| \leq \|uq - \lambda q\| < \epsilon. \qedhere
\]
\end{proof}

The proof of the next result is based on the proof of \cite{Pedersen1979}*{Proposition 8.9.3}.

\begin{lem} \label{lem:derivation-existence}
Let $A$ be an injective C*-algebra and let $\alpha \in \operatorname{Aut}(A)$ be an inner automorphism of $A$. Let $H \leq \operatorname{Aut}(A)$ be a subgroup of automorphisms of $A$ with $\alpha \in Z(H)$ and let $u \in A$ be an $H$-invariant unitary such that $\operatorname{Ad}(u) = \alpha$. Then there is an $H$-invariant open projection $p \in A \cap \{u\}'$ such that $pAp$ is an $H$-invariant essential hereditary C*-subalgebra of $A$, and an $H$-commuting *-derivation $\delta$ of $pAp$ such that $\alpha|_{pAp} = \exp(\delta)$.
\end{lem}

\begin{proof}
For $0 < \epsilon < 1$, it follows from Lemma \ref{lem:derivation-spectrum} and a standard maximality argument that there is a family $\{p_i\}$ of $H$-invariant open projections in $A \cap \{u\}'$ with mutually orthogonal central covers $\{c(p_i)\}$ satisfying $\vee_i c(p_i) = 1$ and a corresponding family $\{\lambda_i\}$ of points in the spectrum of $u$ such that $\|up_i - \lambda_i p_i\| < \epsilon$ for each $i$.

Let $p = \vee_i p_i$. Then $p$ is open and $H$-invariant, and $c(p) = 1$. For $a \in pAp$ with $\|a\| \leq 1$,
\[
\|\alpha(a) - a\| = \|ua - au\| = \sup_i \|p_i u a - a p_i u\| < 2.
\]
Hence $\|\alpha|_{pAp} - \id_{pAp}\| < 2$. It now follows from \cite{Pedersen1979}*{Theorem 8.7.7} that $\delta = \log(\alpha|_{pAp})$ is a *-derivation of $pAp$. Since $\alpha \in Z(H)$ and $p$ is $H$-invariant, $\delta$ is $H$-commuting and $\alpha|_{pAp} = \exp(\delta)$.
\end{proof}

We will also require the following result of Geffen and Ursu, which is contained in the proof of \cite{GeffenUrsu2023}*{Theorem 7.15}. Note that separability is an essential hypothesis here.

\begin{lem} \label{lem:geffen-ursu-unitary}
Let $(A,G,\alpha)$ be a C*-dynamical system with $A$ separable. Consider the following conditions:
\begin{enumerate}
\item There is $r \in FC(G) \setminus \{e\}$, a nonzero $r$-invariant ideal $I \trianglelefteq A$ and a unitary $u \in M(I)$ such that
\begin{enumerate}
\item $I \cap \alpha_g(I)$ is an essential ideal in both $I$ and $\alpha_g(I)$ for all $g \in C_G(r)$ and
\item $C := \|\alpha_r|_I - \operatorname{Ad}(u)\|$ and $D := \sup_{g \in C_G(r)} \|\alpha_g(u) - u\|$ satisfy $2\sqrt{2 - \sqrt{4-C^2}} + D < \sqrt{2}$.
\end{enumerate}
\item There is $r \in FC(G) \setminus \{e\}$, a nonzero $r$-invariant regular ideal $K \trianglelefteq I(A)$ and a $C_G(r)$-invariant unitary $u_r \in K$ such that $\alpha_r|_K = \operatorname{Ad}(u_r)$;
\item $Z(I(A))^G \ne Z(I(A \times_\lambda G))$.
\end{enumerate}
Then (1) is equivalent to (2), and (1) and (2) imply (3).
\end{lem}

\begin{prop} \label{prop:intrinsic-equivalences}
Let $(A,G,\alpha)$ be a C*-dynamical system. The following are equivalent:
\begin{enumerate}
\item $(A,G,\alpha)$ is sub-induced from a regular C*-dynamical subsystem $(J,H,\beta)$ with $r \in FC(H) \setminus \{e\}$ and a $C_H(r)$-invariant essential hereditary C*-subalgebra $B \subseteq J$ such that $\alpha_r|_{B} = \exp(\delta)$ for a $C_H(r)$-commuting *-derivation $\delta$ of $B$.

\item $(I(A),G,\alpha)$ is sub-induced from a regular C*-dynamical subsystem $(J,H,\beta)$ with $r \in FC(H) \setminus \{e\}$ such that $\beta_r = \operatorname{Ad}(u_r)$ for a $C_H(r)$-invariant unitary $u_r \in J$.
\end{enumerate}

If $A$ is separable, then these conditions are additionally equivalent to:
\begin{enumerate}[resume]
\item $(A,G,\alpha)$ is sub-induced from a regular C*-dynamical subsystem $(J,H,\beta)$ with $r \in FC(H) \setminus \{e\}$, a nonzero $r$-invariant ideal $I \trianglelefteq J$ and a unitary $u \in M(I)$ such that
\begin{enumerate}
\item $I \cap \alpha_h(I)$ is an essential ideal in both $I$ and $\alpha_h(I)$ for all $h \in C_H(r)$ and
\item $C := \|\alpha_r|_I - \operatorname{Ad}(u)\|$ and $D := \sup_{h \in C_H(r)} \|\alpha_h(u) - u\|$ satisfy $2\sqrt{2 - \sqrt{4-C^2}} + D < \sqrt{2}$.
\end{enumerate}
\end{enumerate}
\end{prop}

\begin{proof}
We first prove the equivalence between (1) and (2). Suppose (1) holds. By Lemma \ref{lem:imprimitivity-detect-induced-injective-envelope}, $(I(A),G,\alpha)$ is sub-induced from the regular C*-dynamical subsystem $(I(J),H,\beta)$. Arguing as in the proof of \cite{Hamana1985}*{Theorem 7.3} implies the existence of a $C_H(r)$-invariant unitary $u_r \in I(J)$ such that $\beta_r = \operatorname{Ad}(u_r)$, giving (2).

Conversely, suppose (2) holds. By Lemma \ref{lem:derivation-existence}, there is a $C_H(r)$-invariant open projection $p \in J \cap \{u_r\}'$ such that $pJp$ is an $C_H(r)$-invariant essential hereditary C*-subalgebra of $J$, and a $C_H(r)$-commuting *-derivation $\delta$ of $pJp$ such that $\alpha_r|_{pJp} = \exp(\delta)$. Let $I = A \cap J$ and let $B = A \cap pJp$. Then $B$ is a $C_H(r)$-invariant essential hereditary C*-algebra of $I$. Furthermore, $B$ is $\delta$-invariant and $\alpha_r|_B = \exp(\delta|_B)$, giving (1).

Finally, suppose that $A$ is separable. For the equivalence between (2) and (3), first suppose that (3) holds. Then applying Lemma \ref{lem:geffen-ursu-unitary} to the regular C*-dynamical subsystem $(J,H,\beta)$, we obtain (2). Conversely, if (2) holds, then applying the implication (2) implies (1) from Lemma \ref{lem:geffen-ursu-unitary} to the C*-dynamical system $(J,H,\beta)$ implies that (3) holds. 
\end{proof}

Sait{\^o and Wright \cite{SaitoWright1984}*{Theorem 3.6} proved that automorphisms of simple C*-algebras that extend to inner automorphisms on the injective envelope are actually inner. The following result is a slightly stronger version of their result. 

\begin{lem} \label{lem:simple-quasi-inner}
Let $A$ be a simple C*-algebra and let $\alpha \in \operatorname{Aut}(A)$ be an automorphism. Let $H \leq \operatorname{Aut}(A)$ be a subgroup of automorphisms of $A$ with $\alpha \in Z(H)$. Suppose there is an $H$-invariant unitary $v \in I(A)$ such that $\operatorname{Ad}(v) = \alpha$. Then there is an $H$-invariant unitary $u \in M(A)$ such that $\operatorname{Ad}(u) = \alpha$.
\end{lem}

\begin{proof}
Because $A$ is simple, it follows from \cite{SaitoWright1984}*{Theorem 3.6} (see also \cite{SaitoWright1983}*{Theorem 8}) that there is a unitary $u \in M(A)$ such that $\operatorname{Ad}(u) = \alpha$. Note that, since the extension of $\alpha$ to $I(A)$ is unique, $\operatorname{Ad}(u) = \alpha = \operatorname{Ad}(v)$ on $I(A)$. It follows that $uv^* \in Z(I(A))$. But $I(A)$ is prime since $A$ is simple, and hence $Z(I(A)) = \mathbb{C}$. Therefore, $u$ is a scalar multiple of $v$, and in particular is $H$-invariant.
\end{proof}

\section{The regular ideal intersection property} \label{sec:regular-ideal-intersection-property}

In this section we will establish a characterization of the regular ideal intersection property for C*-dynamical systems. In general, the regular ideal intersection property is slightly weaker than the ideal intersection property. But we will see in Section \ref{sec:intersection-property} that these properties coincide for C*-dynamical systems over FC-hypercentral groups.

\begin{defn}
An inclusion $A \subseteq B$ of C*-algebras has the \emph{regular ideal intersection property} if for every nonzero regular ideal $J \trianglelefteq B$, $A \cap J \ne 0$.
\end{defn}

Let $A \subseteq B$ be an inclusion of C*-algebras. Then there is an inclusion $A \subseteq I(B)$, so by the injectivity of $B$ and the essentiality of the inclusion $A \subseteq I(A)$, there is always at least one operator system embedding $I(A) \subseteq I(B)$. However, this embedding is not necessarily unique, and is not necessarily a *-homomorphism (cf. \cite{Hamana1981}*{Remark 3.9}). For the next result, we assume that there is a C*-algebra inclusion of the injective envelopes $I(A) \subseteq I(B)$ that is compatible with the inclusion $A \subseteq B$.

\begin{lem} \label{lem:regular-ideal-intersection-property-iff-injective-envelopes}
Let $A \subseteq B$ be an inclusion of C*-algebras and let $I(A) \subseteq I(B)$ be an inclusion of C*-algebras extending the previous inclusion. The inclusion $A \subseteq B$ has the regular ideal intersection property if and only if the inclusion $I(A) \subseteq I(B)$ does.
\end{lem}

\begin{proof}
Suppose that the inclusion $A \subseteq B$ has the regular ideal intersection property and let $K \trianglelefteq I(B)$ be a nonzero regular ideal. Then by the essentiality of the inclusion $B \subseteq I(B)$, $B \cap K$ is a nonzero regular ideal of $B$. Hence by assumption, $A \cap K \ne 0$, implying $I(A) \cap K \ne 0$.

Conversely, suppose that the inclusion $I(A) \subseteq I(B)$ has the regular ideal intersection property. Let $J \trianglelefteq B$ be a nonzero regular ideal. Then $I(J) \triangleleft I(B)$ is a nonzero regular ideal, so by assumption, $I(A) \cap I(J)$ is a nonzero ideal of $I(A)$. Hence by the essentiality of the inclusion $A \subseteq I(A)$, $A \cap J = A \cap (I(A) \cap I(J)) \ne 0$.
\end{proof}

\begin{lem} \label{lem:monotone-inclusion-projection}
Let $A \subseteq B$ be a normal inclusion of monotone complete C*-algebras and let $J \trianglelefteq B$ be a regular ideal of $B$. Then $A \cap J$ is a regular ideal of $A$. Let $p \in Z(A)$ denote the corresponding central projection, so that $A \cap J = Ap$. If $\beta \in \operatorname{Aut}(B)$ is an automorphism leaving $A$ and $J$ invariant, then $\beta(p) = p$.
\end{lem}

\begin{proof}
Since $J$ is regular, there is a central projection $q \in Z(B)$ such that $J = Bq$. By \cite{Hamana1982a}*{Theorem 6.5}, there is a central projection $p \in Z(A)$ such that $Ap$ is the monotone completion of $A \cap J$ and $p = \sup_A a_i$ for any approximate unit $\{a_i\}$ of $A \cap J$. Since the inclusion $A \subseteq B$ is normal by assumption, $p = \sup_B a_i$, so in particular $p \leq q$, implying $p \in A \cap J$. Hence $A \cap J = Ap$, so $A \cap J$ is regular. If $\beta \in \operatorname{Aut}(B)$ is an automorphism leaving $A$ and $J$ invariant, then it also leaves $A \cap J$ invariant, so it follows that $\beta(p) = p$. 
\end{proof} 

The following result is likely well known, but we were unable to find a reference.

\begin{lem} \label{lem:ideals-from-center}
Let $A$ be a C*-algebra and let $K \trianglelefteq Z(A)$ be an ideal. Then $AK$ is an ideal of $A$ satisfying $Z(A) \cap AK = K$. 
\end{lem}

\begin{proof}
It is clear that $AK$ is an (algebraic) ideal of $A$. The fact that it is closed follows from Cohen's factorization theorem. For $x \in \operatorname{Glimm}(Z(A))$, let $K_x \trianglelefteq Z(A)$ denote the corresponding maximal ideal, i.e. $K_x = \{f \in Z(A) : f(x) = 0\}$. The corresponding Glimm ideal is $AK_x \trianglelefteq A$. The Dauns-Hoffman theorem implies that $AK_x \cap Z(A) = K_x$. Let $C \subseteq \operatorname{Glimm}(Z(A))$ denote the complement of the support of $K$ in $\operatorname{Glimm}(Z(A))$, so that $C$ is a closed subset with $K = \cap_{x \in C} K_x$. Clearly $K \subseteq Z(A) \cap AK$. On the other hand, for every $x \in C$,
\[
Z(A) \cap AK \subseteq Z(A) \cap AK_x = K_x. 
\]
Hence $Z(A) \cap AK \subseteq \cap_{x \in C} K_x = K$, giving $Z(A) \cap AK = K$.
\end{proof}

\begin{prop} \label{prop:necessary-intersection-property}
Let $A \subseteq B$ be a normal inclusion of monotone complete C*-algebras with an inclusion $Z(A) \subseteq Z(B)$. The inclusion $A \subseteq B$ has the regular ideal intersection property if and only if $Z(A) = Z(B)$. 
\end{prop}

\begin{proof}
If $Z(A) = Z(B)$, then since every nonzero regular ideal $J \trianglelefteq B$ is of the form $B = Bq$ for a nonzero central projection $q \in Z(B) = Z(A)$, it follows immediately that the inclusion $A \subseteq B$ has the regular ideal intersection property. 

Conversely, suppose that the inclusion $Z(A) \subseteq Z(B)$ is proper. Since $A$ is monotone complete, $Z(A)$ is injective in the category of unital commutative C*-algebras with *-homomorphisms as morphisms (cf. \cite{Gonshor1968}). Hence there is a *-homomorphism $\pi : Z(B) \to Z(A)$ such that $\pi|_{Z(A)} = \id_{Z(A)}$. Let $C \subseteq \operatorname{Glimm}(Z(B))$ denote the complement of the support of $\ker \pi$. By assumption, $C$ is proper, so there is a non-empty clopen subset $U \subseteq \operatorname{Glimm}(Z(B))$ with $U \cap C = \emptyset$. Let $q = \chi_U \in Z(B)$ denote the corresponding central projection and let $K \trianglelefteq Z(B)$ denote the corresponding ideal, so that $K = Z(B)q$. Then letting $J = BK = Bq$, $J \trianglelefteq B$ is a nonzero proper regular ideal. By Lemma \ref{lem:ideals-from-center},
\[
Z(A) \cap J = Z(A) \cap (Z(B) \cap J) = Z(A) \cap K \subseteq Z(A) \cap \ker \pi = 0.
\]
It follows from Lemma \ref{lem:monotone-inclusion-projection} that $A \cap J = 0$.
\end{proof}

The following result establishes a characterization of the regular ideal intersection property for C*-dynamical systems in terms of injective envelopes. 

\begin{prop} \label{prop:relative-commutant-regular-ideal-intersection-property}
A C*-dynamical system $(A,G,\alpha)$ has the regular ideal intersection property if and only if $Z(I(A))^G = Z(I(A \times_\lambda G))$.
\end{prop}

\begin{proof}
Lemma \ref{lem:regular-ideal-intersection-property-iff-injective-envelopes} implies that $(A,G,\alpha)$ has the regular ideal intersection property if and only if the inclusion $I(A) \subseteq I(A \times_\lambda G)$ has the regular ideal intersection property. Let $J \trianglelefteq I(A \times_\lambda G)$ be a nonzero regular ideal. Then there is a nonzero central projection $q \in Z(I(A \times_\lambda G))$ such that $J = I(A \times_\lambda G)q$. Since the inclusion $I(A) \subseteq I(A \times_\lambda G)$ is normal by Lemma~\ref{lem:sup-preserving}, Lemma \ref{lem:monotone-inclusion-projection} yields the existence of a $G$-invariant central projection $p \in Z(I(A))^G$ such that $I(A)p = I(A) \cap J$. Hence the inclusion $I(A) \subseteq I(A \times_\lambda G)$ has the regular ideal intersection property if and only if the inclusion $Z(I(A))^G \subseteq I(A \times_\lambda G)$ has the regular ideal intersection property. By Lemma \ref{prop:necessary-intersection-property}, this is equivalent to the equality $Z(I(A))^G = Z(I(A \times_\lambda G))$.
\end{proof}

The next result is the main result in this section. It establishes several characterizations of the regular ideal intersection property. Two of these characterizations are intrinsic, and do not refer to injective envelopes.

\begin{thm} \label{thm:regular-ideal-intersection-property}
Let $(A,G,\alpha)$ be a C*-dynamical system. The following are equivalent:
\begin{enumerate}
\item $(A,G,\alpha)$ has the regular ideal intersection property.

\item Whenever $(A,G,\alpha)$ is sub-induced from a regular C*-dynamical subsystem $(J,H,\beta)$ with $r \in H \setminus \{e\}$ and a $C_H(r)$-invariant essential hereditary C*-subalgebra $B \subseteq J$ such that $\alpha_r|_B = \exp(\delta)$ for a $C_H(r)$-commuting *-derivation $\delta$ of $B$, then the $H$-conjugacy class $\operatorname{Cl}_H(r)$ of $r$ is infinite.

\item Whenever $(I(A),G,\alpha)$ is sub-induced from a regular C*-dynamical subsystem $(J,H,\beta)$ with $r \in H \setminus \{e\}$ such that $\beta_r = \operatorname{Ad}(u_r)$ for a $C_H(r)$-invariant unitary $u_r \in J$, then the $H$-conjugacy class $\operatorname{Cl}_H(r)$ of $r$ is infinite.

\item $Z(I(A))^G = Z(I(A \times_\lambda G))$.
\end{enumerate}

If $A$ is separable, then these conditions are additionally equivalent to:
\begin{enumerate}[resume]
\item Whenever $(A,G,\alpha)$ is sub-induced from a regular C*-dynamical subsystem $(J,H,\beta)$ with $r \in H \setminus \{e\}$, a nonzero $r$-invariant ideal $I \trianglelefteq J$ and a unitary $u \in M(I)$ such that
\begin{enumerate}
\item $I \cap \alpha_h(I)$ is an essential ideal in both $I$ and $\alpha_h(I)$ for all $h \in C_H(r)$ and 
\item $C := \|\alpha_r|_I - \operatorname{Ad}(u)\|$ and $D := \sup_{h \in C_H(r)} \|\alpha_h(u) - u\|$ satisfy $2\sqrt{2 - \sqrt{4-C^2}} + D < \sqrt{2}$,
\end{enumerate}
then the $H$-conjugacy class $\operatorname{Cl}_H(r)$ of $r$ is infinite.
\end{enumerate}
\end{thm}

\begin{proof}
The equivalence between (1) and (4) is Proposition \ref{prop:relative-commutant-regular-ideal-intersection-property}, the equivalence between (3) and (4) is Proposition \ref{prop:projection-equivalences} and the equivalence between (2) and (3) is Proposition \ref{prop:intrinsic-equivalences}. If $A$ is separable, then the equivalence between (3) and (5) is also Proposition \ref{prop:intrinsic-equivalences}.
\end{proof}

Specializing Theorem \ref{thm:regular-ideal-intersection-property} to the case of abelian C*-dynamical systems yields the following result.

\begin{cor} \label{cor:abelian-regular-ideal-intersection-property}
Let $(\rC_0(X), G, \alpha)$ be an abelian C*-dynamical system. The following are equivalent:
\begin{enumerate}
\item $(\rC_0(X),G,\alpha)$ has the regular ideal intersection property;
\item Whenever there is a non-empty regular open subset $U \subseteq X$ and a subgroup $H \leq G$ with $H = \{g \in G : gU = U\}$ such that the subsets $\{gU\}_{g \in G/H}$ are pairwise disjoint, then for any $r \in H \setminus \{e\}$ that acts trivially on $U$, the $H$-conjugacy class $\operatorname{Cl}_H(r)$ of $r$ is infinite.
\end{enumerate}
\end{cor}

\begin{proof}
For an abelian C*-dynamical system, an element acts trivially if and only if it is inner. Hence the result follows immediately from Theorem \ref{thm:regular-ideal-intersection-property} and Remark \ref{rem:abelian-sub-induced}.
\end{proof}

Specializing Theorem \ref{thm:regular-ideal-intersection-property} to the case when the underlying C*-algebra is simple yields the following result.

\begin{cor} \label{cor:simple-regular-ideal-intersection-property}
Let $(A,G,\alpha)$ be a C*-dynamical system with $A$ simple. The following are equivalent:
\begin{enumerate}
\item $(A,G,\alpha)$ has the regular ideal intersection property;
\item There is $r \in FC(G) \setminus \{e\}$ and a $C_G(r)$-invariant essential hereditary C*-subalgebra $B \subseteq A$ such that $\alpha_r|_B = \exp(\delta)$ for a $C_G(r)$-commuting *-derivation of $B$;
\item There is $r \in FC(G) \setminus \{e\}$ such that $\alpha_r = \operatorname{Ad}(u_r)$ for a $C_G(r)$-invariant unitary $u_r \in M(A)$. 
\end{enumerate}
\end{cor}

\begin{proof}
The result follows immediately from Theorem \ref{thm:regular-ideal-intersection-property} and Lemma \ref{lem:simple-quasi-inner}.
\end{proof}

In Section \ref{sec:intersection-property}, we will prove that for C*-dynamical systems over FC-hypercentral groups, the regular ideal intersection property always coincides with the ideal intersection property. It is natural to wonder if this could be true for other classes of C*-dynamical systems.

Recall that a topology is said to be {\em semiregular} if it has a basis of regular open subsets. Brown, Fuller, Pitts and Reznikoff \cite{BrownFullerPittsReznikoff2024} proved that if $(A,G,\alpha)$ is a C*-dynamical system such that the primitive ideal space of the reduced crossed product $A \times_\lambda G$ is semiregular, then the regular ideal intersection property for $(A,G,\alpha)$ coincides with the ideal intersection property (see Section~\ref{sec:regular-ideal-intersection-property} and Section~\ref{sec:intersection-property}).

\section{Primality of reduced crossed products} \label{sec:primality}

In this section we will apply our characterization of the regular ideal intersection property for C*-dynamical systems from Section \ref{sec:regular-ideal-intersection-property} to obtain a characterization of primality for reduced crossed product C*-algebras in terms of the underlying C*-dynamical system.

Recall from Section \ref{sec:preliminaries} that a C*-dynamical system $(A,G,\alpha)$ is prime if for every pair of nonzero $G$-invariant ideals $I,J \trianglelefteq A$, the intersection $I \cap J$ is nonzero. By \cite{Hamana1985}*{Proposition 6.4}, the primality of $(A,G,\alpha)$ is equivalent to both the primality of $(I(A),G,\alpha)$ and the primality of $(I_G(A),G,\alpha)$. It is clear that if the reduced crossed product $A \times_\lambda G$ is prime, then $(A,G,\alpha)$ is necessarily prime.

We will require the following result. It follows easily from \cite{Hamana1985}*{Proposition 6.4}, but since it is not explicitly stated there, we provide a proof for convenience.

\begin{lem} \label{lem:G-prime-iff}
A C*-dynamical system $(A,G,\alpha)$ is prime if and only if $Z(I(A))^G = \mathbb{C}$.
\end{lem}

\begin{proof}
If $Z(I(A))^G \ne \mathbb{C}$, then since $Z(I(A))^G$ is monotone complete, there is a non-trivial $G$-invariant projection $p \in Z(I(A))^G$. The ideals $I,J \trianglelefteq A$ defined by $I = A \cap Ap$ and $J = A \cap A(1-p)$ are both nonzero and $G$-invariant, and satisfy $I \cap J = 0$. Hence $(A,G,\alpha)$ is not $G$-prime. 

Conversely, if $(A,G,\alpha)$ is not prime, then there are nonzero $G$-invariant ideals $I,J \trianglelefteq A$ satisfying $I \cap J = 0$. Then $I^{\perp \perp},J^{\perp \perp} \trianglelefteq A$ are nonzero  $G$-invariant regular ideals also satisfying $I^{\perp \perp} \cap J^{\perp \perp} = 0$. Letting $p,q \in Z(I(A))^G$ denote the corresponding $G$-invariant central projections, $p$ and $q$ are orthogonal by Lemma \ref{lem:orthogonal-regular-ideals}. In particular, $Z(I(A))^G \ne \mathbb{C}$.
\end{proof}

\begin{rem}
For prime C*-dynamical systems, being induced and being sub-induced in the sense of Section \ref{sec:imprimitivity} are equivalent. Indeed, it follows from Lemma \ref{lem:G-prime-iff} that if a prime C*-dynamical system $(A,G,\alpha)$ is sub-induced from a regular C*-dynamical subsystem $(J,H,\beta)$, then it is necessarily induced from $(J,H,\beta)$.
\end{rem}

\begin{thm}\label{thm:primeness} Let $(A, G, \alpha)$ be a prime C*-dynamical system. The following are equivalent:
\begin{enumerate}
\item $A \times_\lambda G$ is prime;
\item $(A,G,\alpha)$ has the regular ideal intersection property;
\item Whenever $(A,G,\alpha)$ is induced from a regular C*-dynamical subsystem $(J,H,\beta)$ with $r \in H \setminus \{e\}$ and a $C_H(r)$-invariant essential hereditary C*-subalgebra $B \subseteq J$ such that $\alpha_r|_B = \exp(\delta)$ for a $C_H(r)$-commuting *-derivation $\delta$ of $B$, then the $H$-conjugacy class $\operatorname{Cl}_H(r)$ of $r$ is infinite.
\item $I(A \times_\lambda G)$ is prime;
\item Whenever $(I(A),G,\alpha)$ is induced from a regular C*-dynamical subsystem $(J,H,\beta)$ with $r \in H \setminus \{e\}$ such that $\beta_r = \operatorname{Ad}(u_r)$ for a $C_H(r)$-invariant unitary $u_r \in J$, then the $H$-conjugacy class $\operatorname{Cl}_H(r)$ of $r$ is infinite;
\item $Z(I(A \times_\lambda G)) = \mathbb{C}$.
\end{enumerate}
If $A$ is separable, then these conditions are additionally equivalent to:
\begin{enumerate}[resume]
\item Whenever $(A,G,\alpha)$ is induced from a regular C*-dynamical subsystem $(J,H,\beta)$ with $r \in H \setminus \{e\}$, a nonzero $r$-invariant ideal $I \trianglelefteq J$ and a unitary $u \in M(I)$ such that
\begin{enumerate}
\item $I \cap \alpha_h(I)$ is an essential ideal in both $I$ and $\alpha_h(I)$ for all $h \in C_H(r)$ and 
\item $C := \|\alpha_r|_I - \operatorname{Ad}(u)\|$ and $D := \sup_{h \in C_H(r)} \|\alpha_h(u) - u\|$ satisfy $2\sqrt{2 - \sqrt{4-C^2}} + D < \sqrt{2}$,
\end{enumerate}
then the $H$-conjugacy class $\operatorname{Cl}_H(r)$ of $r$ is infinite.
\end{enumerate}
\end{thm}

\begin{proof}
The equivalence between (1), (4) and (6) follows from the fact that a C*-algebra is prime if and only if its injective envelope is prime, or equivalently, has trivial center (cf. Theorem 6.3 and Theorem 7.1 of \cite{Hamana1981}).

Since $(A,G,\alpha)$ is prime, Lemma \ref{lem:G-prime-iff} implies $Z(I(A))^G = \mathbb{C}$. It follows from this and Theorem \ref{thm:regular-ideal-intersection-property} that (2), (3), (5) and (6) are equivalent. If $A$ is separable, then Theorem \ref{thm:regular-ideal-intersection-property} implies that these conditions are additionally equivalent to (7). 
\end{proof}

\begin{rem} \label{rem:recover-geffen-ursu-primality}
If $(A,G,\alpha)$ is minimal, meaning that $A$ has no non-trivial $G$-invariant ideals, then the equivalence between (1) and (5) recovers \cite{GeffenUrsu2023}*{Theorem 6.2}. To see this, first suppose that $(I(A),G,\alpha)$ is induced from a regular C*-dynamical subsystem $(J,H,\beta)$ with $r \in FC(H) \setminus \{e\}$ such that $\beta_r = \operatorname{Ad}(u_r)$ for a $C_H(r)$-invariant unitary $u_r \in J$. Note that $H$ necessarily has finite index in $G$, so $FC(H) \leq FC(G)$. Letting $T \subseteq C_G(r)$ be a left transversal of $C_H(r)$ in $C_G(r)$, $T$ is finite. Let $p \in Z(I(A))$ denote the central projection corresponding to $J$, so that $J = I(A)p$. Then letting $q = \sum_{g \in T} \alpha_g(p)$, and $v = \sum_{g \in T} \alpha_g(u)$, $q$ is a nonzero $C_G(r)$-invariant central projection in $I(A)$, $v$ is a $C_G(r)$-invariant unitary in $I(A)q$ and $\alpha_r|_{I(A)q} = \operatorname{Ad}(v)$

For the converse, suppose there is $r \in FC(G) \setminus \{e\}$, a nonzero $C_G(r)$-invariant central projection $q \in Z(I(A))$ and a unitary $v \in I(A)q$ such that $\alpha_r|_{I(A)q} = \operatorname{Ad}(v)$. Then $C_G(r)$ has finite index in $G$, so $q$ is a meadering projection for $(A,G,\alpha)$ in the sense of Definition \ref{defn:meandering-projection}. It follows from Proposition \ref{prop:projection-equivalences} and Theorem \ref{thm:primeness} that $A \times_\lambda G$ is not prime.
\end{rem}

Recall that an abelian C*-dynamical system $(\rC_0(X),G,\alpha)$ is prime if and only if every non-empty invariant open subset is dense. Specializing Theorem \ref{thm:primeness} to the case of abelian C*-dynamical systems yields the following result, the proof of which is essentially the same as the proof of Corollary \ref{cor:abelian-regular-ideal-intersection-property}.

\begin{cor} \label{cor:abelian-primeness}
Let $(\rC_0(X), G, \alpha)$ be a prime abelian C*-dynamical system. The following are equivalent:
\begin{enumerate}
\item $\rC_0(X) \times_\lambda G$ is prime;
\item Whenever there is a non-empty regular open subset $U \subseteq X$ and a subgroup $H \leq G$ with $H = \{g \in G : gU = U\}$ such that the subsets $\{gU\}_{g \in G/H}$ are pairwise disjoint, then for any $r \in H \setminus \{e\}$ that acts trivially on $U$, the $H$-conjugacy class $\operatorname{Cl}_H(r)$ of $r$ is infinite.
\end{enumerate}
\end{cor}

Corollary \ref{cor:abelian-primeness} should be compared to Vaes' characterization of factoriality for crossed products of abelian W*-dynamical systems \cite{Vaes2020}.

If $(A,G,\alpha)$ is a C*-dynamical system with $A$ simple, then it is trivially prime. Specializing Theorem \ref{thm:primeness} to this case yields the following result, the proof of which is identical to the proof of Corollary \ref{cor:simple-regular-ideal-intersection-property}.

\begin{cor} \label{cor:simple-primeness}
Let $(A,G,\alpha)$ be a C*-dynamical system with $A$ simple. The following are equivalent:
\begin{enumerate}
\item $A \times_\lambda G$ is prime;
\item There is $r \in FC(G) \setminus \{e\}$ and a $C_G(r)$-invariant essential hereditary C*-subalgebra $B \subseteq A$ such that $\alpha_r|_B = \exp(\delta)$ for a $C_G(r)$-commuting *-derivation of $B$;
\item There is $r \in FC(G) \setminus \{e\}$ such that $\alpha_r = \operatorname{Ad}(u_r)$ for a $C_G(r)$-invariant unitary $u_r \in M(A)$.
\end{enumerate}
\end{cor}

\section{Pseudo-expectations} \label{sec:pseudo-expectations}

In this section we will relate the regular ideal intersection property from Section \ref{sec:regular-ideal-intersection-property} to an important class of maps on reduced crossed products called pseudo-expectations. This notion was introduced in \cite{KennedySchafhauser2019} in order to study the ideal intersection property, and can be seen as an equivariant generalization of Pitts' notion of pseudo-expectation from \cite{Pitts2017}.

Recall that if $(A,G,\alpha)$ is a C*-dynamical system, then there is always a conditional expectation $E : A \times_\lambda G \to A$ satisfying
\[
E(a \lambda_g) = a \delta_e(g), \quad a \in A,\ g \in G.
\]
The conditional expectation $E$ is the prototypical example of a pseudo-expectation. Note that $E$ is $G$-equivariant. Furthermore, identifying $A$ with its image in $I_G(A)$, we can view $E$ as a map from $A \times_\lambda G$ to $I_G(A)$.

\begin{defn}
Let $(A,G,\alpha)$ be a C*-dynamical system. A {\em pseudo-expectation} for $(A,G,\alpha)$ is a $G$-equivariant ucp map $\phi : A \times_\lambda G \to I_G(A)$ with the property that $\phi|_A = \id_A$. The conditional expectation $E$ above is the {\em canonical pseudo-expectation} on $A \times_\lambda G$. 
\end{defn}

\begin{rem}
Note that a $G$-equivariant conditional expectation $\phi : I_G(A) \times_\lambda G \to I_G(A)$ restricts to a pseudo-expectation for $(A,G,\alpha)$. Conversely, if $\phi : A \times_\lambda G \to I_G(A)$ is a pseudo-expectation for $(A,G,\alpha)$, then it follows from the essentiality of the inclusion $(A,G,\alpha) \subseteq (I_G(A),G,\alpha)$ that $\phi$ has a unique extension to a $G$-equivariant conditional expectation $\phi : I_G(A) \times_\lambda G \to I_G(A)$. This gives a bijective correspondence between pseudo-expectations for $(A,G,\alpha)$ and $G$-equivariant conditional expectations from $I_G(A) \times_\lambda G$ to $I_G(A)$. 
\end{rem}

\begin{defn}
A C*-dynamical system $(A,G,\alpha)$ has the {\em unique pseudo-expectation property} if the only pseudo-expectation for $(A,G,\alpha)$ is the canonical one.
\end{defn}

The idea underlying the proof of the following result originates in work of B\'{e}dos and Omland \cite{BedosOmland2016} on twisted group C*-algebras. This idea was also utilized by Geffen and Ursu \cite{GeffenUrsu2023}*{Lemma 5.1}, and our proof is similar to theirs. Since the statement of our result is slightly different and we do not require minimality, we will provide a proof for convenience.

\begin{lem} \label{lem:pseudo-expectation-bounded-sum}
Let $(A,G,\alpha)$ be a C*-dynamical system and let $\phi : A \times_\lambda G \to I_G(A)$ be a pseudo-expectation for $(A,G,\alpha)$. Let $N \trianglelefteq G$ be a normal subgroup such that the restriction $\phi|_{A \times_\lambda N}$ agrees with the canonical pseudo-expectation for $(A,N,\alpha|_N)$. Then for $gN \in G/N$,
\[
\sup_F \textstyle{\| \sum_{s \in F} \phi(\lambda_s)^* \phi(\lambda_s) \|} \leq 1,
\]
where the supremum is taken over all finite subsets $F \subseteq gN$.
\end{lem}

\begin{proof}
Fix a finite subset $F \subseteq gN$ and let
\[
x = 1 - \sum_{s \in F} \phi(\lambda_s)^*\lambda_s.
\]
Then
\[
xx^* = 1 - \sum_{s \in F} \phi(\lambda_s)^* \lambda_s - \sum_{t \in F} \lambda_t^* \phi(\lambda_t) + \sum_{s \in F} \sum_{t \in F} \phi(\lambda_s)^* \lambda_{st^{-1}} \phi(\lambda_t).
\]
For $s \in F$, using the fact that $I_G(A)$ belongs to the multiplicative domain of $\phi$ gives
\[
\phi(\phi(\lambda_s)^* \lambda_s) = \phi(\lambda_s)^* \phi(\lambda_s) = \phi(\lambda_s^* \phi(\lambda_s)).
\]
Also, for $s,t \in F$ with $s \ne t$, $st^{-1} \in N \setminus \{e\}$, so a similar computation, using the assumption that $\phi|_{A \times_\lambda N}$ restricts to the canonical pseudo-expectation for $(A,N,\alpha|_N)$, gives
\[
\phi(\phi(\lambda_s)^* \lambda_{st^{-1}} \phi(\lambda_t)) = \phi(\lambda_s)^* \phi(\lambda_{st^{-1}}) \phi(\lambda_t) = 0.
\]
Applying $\phi$ and the above equalities to the above expression for $xx^*$ gives
\[
0 \leq \phi(xx^*) = 1 - \sum_{s \in F} \phi(\lambda_s)^* \phi(\lambda_s). \qedhere
\]
\end{proof}

\begin{lem} \label{lem:pseudo-expectation-r-projection}
Let $(A,G,\alpha)$ be a C*-dynamical system and let $\phi : A \times_\lambda G \to I_G(A)$ be a pseudo-expectation for $(A,G,\alpha)$. Let $N \trianglelefteq G$ be a normal subgroup such that the restriction $\phi|_{A \times_\lambda N}$ agrees with the canonical pseudo-expectation for $(A,N,\alpha|_N)$. Suppose there is $r \in G \setminus \{e\}$ such that $rN \in FC(G/N)$ and $\phi(\lambda_r) \ne 0$. Then $(A,G,\alpha)$ does not have the regular ideal intersection property.
\end{lem}

\begin{proof}
We will prove that there is a meandering projection for $(I_G(A),G,\alpha)$ in the sense of Definition \ref{defn:meandering-projection}. The result will then follow from Proposition \ref{prop:projection-equivalences} and Theorem \ref{thm:regular-ideal-intersection-property}.

Since $rN \in FC(G/N)$, there is a finite subset $E \subseteq \operatorname{Cl}_G(r)$ such that $\operatorname{Cl}_G(r) \subseteq E N$. It follows from Lemma \ref{lem:pseudo-expectation-bounded-sum} that
\[
\sup_F \left\| \sum_{g \in F} \phi(\lambda_g)^* \phi(\lambda_g) \right\| \leq \left\| \sum_{s \in E} \sum_{t \in sN} \phi(\lambda_t)^*\phi(\lambda_t) \right\| \leq |E| < \infty,
\]
where the supremum is taken over all finite subsets $F \subseteq \operatorname{Cl}(r)$.

For $g \in G$, let $x_g = \phi(\lambda_g)$. Then since $A$ belongs to the multiplicative domain of $\phi$, $x_g a = \alpha_g(a) x_g$ for $a \in A$, i.e. $x_g$ is a dependent element for $r$ in the terminology of Kallman \cite{Kallman1969}. By taking the polar decomposition of $x_g$, we obtain a central projection $p_g \in Z(I_G(A))$ and a unitary $u_g \in I_G(A)p_r$ such that $\alpha_g|_{I_G(A)p_g} = \operatorname{Ad}(u_g)$ (see e.g. \cite{GeffenUrsu2023}*{Proposition 3.8} for details). Furthermore, it follows from the $G$-equivariance of $\phi$ that $\alpha_h(p_g) = p_{hgh^{-1}}$ and $\alpha_h(u_g) = u_{hgh^{-1}}$ for all $h \in G$.

Let $f = x_r^*x_r$. Then it follows from above that $0 \ne f \in Z(I_G(A))$. Choose $\lambda > 0$ with $2\lambda$ in the range of $f$ and let $U \subseteq \operatorname{Glimm}(I_G(A))$ denote the closure of the non-empty open set
\[
\{x \in \operatorname{Glimm}(I_G(A)) : f(x) > \lambda\}.
\]
Since $\operatorname{Glimm}(I_G(A))$ is extremally disconnected, $U$ is clopen. Also, $U$ is $C_G(r)$-invariant since $f$ is. Let $p = \chi_U \in Z(I_G(A))$ denote the corresponding central projection. Then $p \ne 0$ and $p$ is $C_G(r)$-invariant. We will show that $p$ is meandering for $(I_G(A),G,\alpha)$.

Let $u = u_r p$. Then $u \in I_G(A)p$ and since $p \leq p_r$, $\alpha_r|_{I_G(A)p} = \operatorname{Ad}(u)$. Furthermore, from above $u$ is $C_G(r)$-invariant. Finally, let $T \subseteq G$ be a left transversal of $C_G(r)$ in $G$. Then for a finite subset $F \subseteq T$, it follows from above that
\[
\left\| \sum_{t \in F} \alpha_t(p) \right\| \leq \lambda^{-1} \left\| \sum_{t \in F} \alpha_t(f) \right\| = \lambda^{-1} \left\| \sum_{t \in F} \phi(\lambda_{trt^{-1}})^* \phi(\lambda_{trt^{-1}}) \right\| \leq \lambda^{-1} |E|.
\]
Hence $p$ is meandering for $(I_G(A),G,\alpha)$.
\end{proof}

\section{The ideal intersection property} \label{sec:intersection-property}

In this section we will prove that a C*-dynamical system over an FC-hypercentral group has the ideal intersection property if and only if it has the regular ideal intersection property. Combined with our characterization of the regular ideal intersection property in Section \ref{sec:regular-ideal-intersection-property}, this will yield a characterization of the ideal intersection property for C*-dynamical systems over FC-hypercentral groups.

We will see that for prime C*-dynamical systems over FC-hypercentral groups, the ideal intersection property is equivalent to the primality of the corresponding reduced crossed product. This extends a result of Echterhoff \cite{Echterhoff1990}*{Theorem 3.1} for minimal C*-dynamical systems over FC-hypercentral groups.

\begin{defn}
A C*-dynamical system $(A,G)$ has the \emph{ideal intersection property} if every nonzero ideal of $A \times_{\lambda} G$ has nonzero intersection with $A$. 
\end{defn}

The following result relates the ideal intersection property to the theory of pseudo-expectations from Section~\ref{sec:pseudo-expectations}. We will utilize this result in our characterization of the ideal intersection property for FC-hypercentral groups.

\begin{prop} \label{prop:regular-ideal-intersection-property-implies-unique-pseudo-expectation-property}
Let $(A,G,\alpha)$ be a C*-dynamical system. If $(A,G,\alpha)$ has the unique pseudo-expectation property, then it has the ideal intersection property. If $G$ is FC-hypercentral and $(A,G,\alpha)$ has the regular ideal intersection property, then it has the unique pseudo-expectation property, and hence the ideal intersection property.
\end{prop}

\begin{proof}
The fact that $(A,G,\alpha)$ has the ideal intersection property if it has the unique pseudo-expectation property follows from \cite{KennedySchafhauser2019}*{Theorem 6.6}.

Conversely, suppose that $G$ is FC-hypercentral and that $(A,G,\alpha)$ has the regular ideal intersection property. Let $\phi : A \times_\lambda G \to I_G(A)$ be a pseudo-expectation for $(A,G,\alpha)$. We must prove that $\phi$ agrees with the canonical pseudo-expectation for $(A,G,\alpha)$, which is equivalent to proving that $\phi(\lambda_g) = 0$ for all $g \in G \setminus \{e\}$.

Let $(F_i)_i$ be the FC-central series of $G$, so that $F_1 = FC(G)$ and $G = \cup_i F_i$. It will be convenient to let $F_0$ denote the trivial group. It suffices to prove that for each $i$, the restriction $\phi|_{A \times_\lambda F_i}$ agrees with the canonical pseudo-expectation for $(A,F_i,\alpha|_{F_i})$. We will proceed via transfinite induction. 

The above statement is trivially true for $i = 0$. For a successor ordinal $i+1$, suppose that $\phi|_{A \times_\lambda F_i}$ agrees with the canonical-pseudo-expectation for $(A,F_i,\alpha|_{F_i})$. For $r \in F_{i+1} \setminus F_i$, $F_i$ is normal in $G$ and $rF_i \in FC(G/Fi)$. Therefore, if $\phi(\lambda_r) \ne 0$, then Lemma \ref{lem:pseudo-expectation-r-projection} implies that $(A,G,\alpha)$ does not have the regular ideal intersection property, which would be a contradiction. Hence $\phi(\lambda_r) = 0$ for all $r \in F_{i+1}$, so $\phi|_{A \times_\lambda F_{i+1}}$ agrees with the canonical pseudo-expectation for $(A,F_{i+1},\alpha|_{F_{i+1}})$.

Finally, if $j$ is a limit ordinal and $\phi|_{A \times_\lambda F_i}$ agrees with the canonical pseudo-expectation for $(A,F_i,\alpha|_{F_i})$ for all $i < j$, then since $F_j = \cup_{i < j} F_i$, it is clear that $\phi|_{A \times_\lambda F_j}$ agrees with the canonical pseudo-expectation for $(A,F_j,\alpha|_{F_j})$.
\end{proof}

The next result is our characterization of the ideal intersection property for C*-dynamical systems over FC-hypercentral groups.

\begin{thm} \label{thm:FC-hypercentral-ideal-intersection-property}
Let $(A,G,\alpha)$ be a C*-dynamical system with $G$ FC-hypercentral. The following are equivalent:
\begin{enumerate}
\item $(A,G,\alpha)$ has the ideal intersection property;

\item $(A,G,\alpha)$ has the regular ideal intersection property;

\item $(A,G,\alpha)$ has the unique pseudo-expectation property;

\item Whenever $(A,G,\alpha)$ is sub-induced from a regular C*-dynamical subsystem $(J,H,\beta)$ with $r \in H \setminus \{e\}$ and a $C_H(r)$-invariant essential hereditary C*-subalgebra $B \subseteq J$ such that $\alpha_r|_B = \exp(\delta)$ for a $C_H(r)$-commuting *-derivation $\delta$ of $B$, then the $H$-conjugacy class $\operatorname{Cl}_H(r)$ of $r$ is infinite;

\item Whenever $(I(A),G,\alpha)$ is sub-induced from a regular C*-dynamical subsystem $(J,H,\beta)$ with $r \in H \setminus \{e\}$ such that $\beta_r = \operatorname{Ad}(u_r)$ for a $C_H(r)$-invariant unitary $u_r \in J$, then the $H$-conjugacy class $\operatorname{Cl}_H(r)$ of $r$ is infinite;

\item $Z(I(A))^G = Z(I(A \times_\lambda G))$.
\end{enumerate}
If, in addition, $A$ is separable, then these conditions are equivalent to:
\begin{enumerate}[resume]
\item Whenever $(A,G,\alpha)$ is sub-induced from a regular C*-dynamical subsystem $(J,H,\beta)$ with $r \in H \setminus \{e\}$, a nonzero $r$-invariant ideal $I \trianglelefteq J$ and a unitary $u \in M(I)$ such that
\begin{enumerate}
\item $I \cap \alpha_h(I)$ is an essential ideal in both $I$ and $\alpha_h(I)$ for all $h \in C_H(r)$ and 
\item $C := \|\alpha_r|_I - \operatorname{Ad}(u)\|$ and $D := \sup_{h \in C_H(r)} \|\alpha_h(u) - u\|$ satisfy $2\sqrt{2 - \sqrt{4-C^2}} + D < \sqrt{2}$,
\end{enumerate}
then the $H$-conjugacy class $\operatorname{Cl}_H(r)$ of $r$ is infinite.
\end{enumerate}

\end{thm}

\begin{proof}
The equivalence between (1), (2) and (3) follows from Proposition \ref{prop:regular-ideal-intersection-property-implies-unique-pseudo-expectation-property}. The equivalence between (2), (4), (5), and (6) Theorem \ref{thm:regular-ideal-intersection-property}. If $A$ is separable, then Theorem \ref{thm:regular-ideal-intersection-property} additionally implies that these conditions are equivalent to (7).
\end{proof}

\begin{cor}
Let $(A,G,\alpha)$ be a prime C*-dynamical system with $G$ FC-hypercentral. Then $(A,G,\alpha)$ has the intersection property if and only if $A \times_\lambda G$ is prime.
\end{cor}

\begin{proof}
Theorem \ref{thm:FC-hypercentral-ideal-intersection-property} implies that $(A,G,\alpha)$ has the ideal intersection property if and only if it has the regular ideal intersection property. Since $(A,G,\alpha)$ is prime, Theorem \ref{thm:primeness} implies that this is equivalent to the primality of $A \times_\lambda G$.
\end{proof}

Kawabe \cite{Kawabe2017}*{Theorem 1.6} obtained a characterization of the ideal intersection property for abelian C*-dynamical systems over arbitrary discrete groups in terms of the dynamics of stabilizer subgroups. For the class of abelian C*-dynamical systems over amenable groups, which includes the class of FC-hypercentral groups, Kawamura and Tomiyama \cite{KawamuraTomiyama1990}*{Theorem 4.1} proved that the ideal intersection property is equivalent to the topological freeness of the action. Specializing Theorem \ref{thm:FC-hypercentral-ideal-intersection-property} to the case of abelian C*-dynamical systems over FC-hypercentral groups yields the following alternative characterization. 

\begin{cor}
    Let $(\rC_0(X), G, \alpha)$ be an abelian C*-dynamical system with $G$ FC-hypercentral. The following are equivalent:
    \begin{enumerate}
    \item $(\rC_0(X), G, \alpha)$ has the ideal intersection property;
    \item Whenever there is a non-empty regular open subset $U \subseteq X$ and a subgroup $H \leq G$ with $H = \{g \in G : gU = U\}$ such that the subsets $\{gU\}_{g \in G/H}$ are pairwise disjoint, then for any $r \in H \setminus \{e\}$ that acts trivially on $U$, the $H$-conjugacy class $\operatorname{Cl}_H(r)$ of $r$ is infinite.
    \end{enumerate}
\end{cor}

\begin{proof}
The result follows immediately from Theorem \ref{thm:FC-hypercentral-ideal-intersection-property} and Corollary~\ref{cor:abelian-regular-ideal-intersection-property}.
\end{proof}

Specializing Theorem \ref{thm:FC-hypercentral-ideal-intersection-property} to the case when the underlying C*-algebra is simple yields the following result.

\begin{cor}
Let $(A,G,\alpha)$ be a C*-dynamical system with $A$ simple and $G$ FC-hypercentral. The following are equivalent:
\begin{enumerate}
\item $(A,G,\alpha)$ has the ideal intersection property;
\item There is $r \in FC(G) \setminus \{e\}$ and a $C_G(r)$-invariant essential hereditary C*-subalgebra $B \subseteq A$ such that $\alpha_r|_B = \exp(\delta)$ for a $C_G(r)$-commuting *-derivation of $B$;
\item There is $r \in FC(G) \setminus \{e\}$ such that $\alpha_r = \operatorname{Ad}(u_r)$ for a $C_G(r)$-invariant unitary $u_r \in M(A)$. 
\end{enumerate}
\end{cor}

\begin{proof}
The proof follows immediately from Theorem \ref{thm:FC-hypercentral-ideal-intersection-property} and Corollary \ref{cor:simple-regular-ideal-intersection-property}. 
\end{proof}

\begin{example} \label{ex:non-fc-case}
For non-FC-hypercentral groups, the regular ideal intersection property does not necessarily coincide with the ideal intersection property. To see this, consider the amenable group $S_\infty$ consisting of all finitely supported permutations of the natural numbers. This group is ICC, and it is well known that the reduced C*-algebra of an ICC group is prime. Hence the reduced C*-algebra $\mathrm{C}_\lambda^*(S_\infty)$, is prime. Since $\mathrm{C}_\lambda^*(S_\infty)$ is also the reduced crossed product of the trivial C*-dynamical system $(\mathbb{C},S_\infty,\alpha)$, Theorem \ref{thm:regular-ideal-intersection-property} implies that $(\mathbb{C},S_\infty,\alpha)$ has the regular ideal intersection property. However, since $S_\infty$ is amenable, $\mathrm{C}_\lambda^*(S_\infty) = \mathrm{C}^*(S_\infty)$, where $\mathrm{C}^*(S_\infty)$ denotes the full C*-algebra of $S_\infty$. In particular, the trivial representation of $S_\infty$ extends to a representation of $\mathrm{C}_\lambda^*(S_\infty)$ with nonzero kernel. Hence $\mathrm{C}_\lambda^*(S_\infty)$ is not simple, so $(\mathbb{C},S_\infty,\alpha)$ does not have the ideal intersection property.
\end{example}

\section{Examples and applications} \label{sec:examples}

In this section we will present examples and applications of our main results. For many groups, the subgroup structure is restrictive enough that the conditions in these results become significantly easier to verify. We will demonstrate this for free products of cyclic groups and, in particular, the projective special linear group $\mathrm{PSL}_2(\mathbb{Z})$, as well as for the special linear group $\mathrm{SL}_2(\mathbb{Z})$. 

\subsection{First examples}

We begin with an important example illustrating that the invariance of the unitary required in Condition (5) of Theorem \ref{thm:primeness} is generally necessary. This example also appeared in \cite{KennedySchafhauser2019}*{Example 5.6} and \cite{GeffenUrsu2023}*{Example 9.4}.

\begin{example} \label{ex:cant-drop-invariance-condition}
 Let $A = M_2(\mathbb{C})$ and $G = \mathbb{Z}_2^2 = \{e,u,v,uv\}$. Define an action $\alpha : G \to \operatorname{Aut}(A)$ by
\[
\alpha_e = \id_A,\ \alpha_u = \operatorname{Ad}(U),\ \alpha_v = \operatorname{Ad}(V),\ \alpha_{uv} = \operatorname{Ad}(UV),
\]
where $U,V \in A$ are the unitaries
\[
U = \left[\begin{matrix} 1 & 0 \\ 0 & -1 \end{matrix}\right],\ V = \left[\begin{matrix} 0 & 1 \\ 1 & 0 \end{matrix}\right].
\]
It is easy to check that $A \times_\lambda G \cong M_4(\mathbb{C})$, which is simple, so in particular $(A,G,\alpha)$ has the ideal intersection property. 

Note that $A = I(A)$ and $A$ is simple. By definition, the action of every non-trivial element in $FC(G) = G$ is inner. However, $C_G(u) = G$, and the unitary $U$ is not $G$-invariant. Similarly, $V$ and $UV$ are not $G$-invariant. This demonstrates the necessity of the (approximate) invariance condition on the unitary in conditions (5) and (6) in Theorem \ref{thm:FC-hypercentral-ideal-intersection-property}.
\end{example}

The following example appeared in \cite{Ursu2022}*{Example 4.4.1}, where it was used to illustrate that the W*-crossed product of an ergodic abelian W*-dynamical system over an ICC group is not necessarily a factor. We will apply Theorem \ref{thm:primeness} to show that the corresponding reduced crossed product is not prime. Consequently, the W*-dynamical system does not have the regular ideal intersection property. In contrast, the reduced crossed product of a minimal C*-dynamical system over an ICC group is always prime by \cite{GeffenUrsu2023}*{Corollary 6.4}.

\begin{example} \label{ex:ursu} Let $G$ be an infinite group with only two conjugacy classes (see e.g. \cite{HigmanNeumannHeuman1949}*{Theorem III}). Then $G$ is ICC. Let $A\coloneqq \ell^\infty(G\setminus\{e\})$. Since $A$ is an abelian W*-algebra, it is monotone complete, so Lemma \ref{lem:equivalence-abelian-injectivity} implies that it is injective. Define $\alpha : G \to \operatorname{Aut}(A)$ by
\[
\alpha_g(f)(t) = f(gtg^{-1}), \quad f\in A,\ t\in G.
\]
It follows from the properties of $G$ that  $(A, G, \alpha)$ is a prime C*-dynamical system. However, it is not minimal, since $\mathrm{c}_0(G \setminus \{e\})$ is a non-trivial $G$-invariant ideal of $A$. 

For $r \in G \setminus \{e\}$, let $p := \chi_{\{r\}} \in A$, where $\chi_{\{r\}}$ denotes the indicator function for the singleton $\{r\}$. We can also view $p$ as a unitary in $Ap$, and $\alpha_r|_{Ap} = \operatorname{Ad}(p)$. Consider the regular C*-dynamical subsystem $(Ap,C_G(r),\beta)$, where $\beta$ is obtained by restricting $\alpha$ in the obvious way. It is easy to verify that $(A,G,\alpha)$ is induced from $(Ap,C_G(r),\beta)$ in the sense of Definition \ref{defn:induced}. Clearly $r \in FC(C_G(r))$. Therefore, by Theorem \ref{thm:primeness}, the reduced crossed product $A \times_\lambda G$ is not prime. 

Theorem \ref{thm:imprimitivity-reduced-crossed-product} yields isomorphisms
\[
A \cong \ell^\infty(G/C_G(r)), \quad I(A \times_\lambda G) \cong \mathcal{B}(\ell^2(G/C_G(r))) \overline{\otimes} I(\mathrm{C}^*_\lambda(C_G(r))).
\]
The *-derivation $\delta$ on $Ap \cong \mathbb{C}$, obtained from Condition (3) of Theorem \ref{thm:primeness} and satisfying $\exp(\delta) = \id_{Ap}$, is simply zero.
\end{example} 

    Recall that a group is C*-simple if its reduced C*-algebra is simple. Bryder and the first author proved in \cite{BryderKennedy2018}*{Theorem 4.3} that if $(A,G,\alpha)$ is a C*-dynamical system with $A$ unital and $G$ C*-simple, then for a maximal ideal $J \trianglelefteq A \times_\lambda G$, the intersection $A \cap I$ is a maximal $G$-invariant ideal of $A$, and moreover that this correspondence is a bijection. In particular, if $(A,G,\alpha)$ is minimal, then $A \times_\lambda G$ is simple, which recovers the result from \cite{BreuillardKalantarKennedyOzawa2017}*{Theorem 7.1} that reduced crossed products of minimal C*-dynamical systems over C*-simple groups are simple, and in particular prime. 

    However, as the following example demonstrates, if $(A,G,\alpha)$ is a non-prime C*-dynamical system with $G$ C*-simple, then it is not necessarily true that $(A,G,\alpha)$ is prime, even if $A$ is unital. This example was motivated by \cite{delaHarpeSkandalis1986}.
    
    \begin{example}
    Let $G$ be a C*-simple group and let $H \leq G$ be a non-trivial abelian subgroup. For example, take $G = \mathbb{F}_2 = \langle a,b \rangle$, the free group on two generators, and take $H = \langle a \rangle \cong \mathbb{Z}$. Let $A$ denote the minimal unitization of $\mathrm{c}_0(G/H)$. Define $\alpha : G \to \operatorname{Aut}(A)$ by extending the left translation action and consider the C*-dynamical system $(A,G,\alpha)$. Since $A$ has a unique non-trivial $G$-invariant ideal, namely $\mathrm{c}_0(G/H)$, $(A,G,\alpha)$ is prime.
    
    It follows from Lemma \ref{lem:easy-inj-env} that $I(A) = \ell^\infty(G/H)$. Since $(A,G,\alpha)$ is prime, it follows from \cite{Hamana1985}*{Proposition 6.4} that $(I(A),G,\alpha)$ is prime. Let $J$ denote the regular ideal in $\ell^\infty(G/H)$ consisting of elements supported on the singleton $\{H\}$. Then $J$ is $H$-invariant, so we obtain a regular C*-dynamical subsystem $(J,H,\beta)$ of $(I(A),G,\alpha)$ by restricting $\alpha$. It is clear that $(I(A),G,\alpha)$ is induced by $(J,H,\beta)$. Since the action on $J$ is trivial and $H$ is abelian, it follows from Theorem \ref{thm:primeness} that the reduced crossed product $A \times_\lambda G$ is not prime. 
    
    Theorem \ref{thm:imprimitivity-reduced-crossed-product} implies that
    \[
    I(A \times_\lambda G) \cong \mathcal{B}(\ell^2(G/H)) \overline{\otimes} I(\mathrm{C}_\lambda^*(H)). 
    \]
    It is straightforward, using this decomposition of $I(A \times_\lambda G)$, combined with the fact that $H$ is non-trivial and abelian, to exhibit two non-trivial ideals of $I(A \times_\lambda G)$ with trivial intersection.
    \end{example}

\subsection{Cyclic centralizers} \label{sec:cyclic-subgroups}

In this subsection, we will establish conditions on elements in the FC-center of a group implying that for any C*-dynamical system over the group with the regular ideal intersection property, the element necessarily acts properly outer. These conditions will be in terms of the index and the centralizer of the cyclic subgroup generated by the element. In particular, by Theorem \ref{thm:primeness}, this result will apply when the corresponding reduced crossed product is prime. 

For many groups, these conditions will be satisfied by most or even all of the elements of the group. Thus for C*-dynamical systems over these groups, this will lead to a very simple characterization of the regular ideal intersection property. We will demonstrate this in later subsections for the groups $\mathrm{PSL}_2(\ZZ)$ and $\mathrm{SL}_2(\ZZ)$, as well as for free products of cyclic groups. Additionally, as an application in Corollary \ref{cor:cyclic-groups} below, we will give a simple proof that for C*-dynamical systems over $\ZZ$ or $\ZZ_n$, with $n$ a square-free positive integer, the ideal intersection property is equivalent to the proper outerness of the system.

This latter result is not new. For separable C*-dynamical systems over $\mathbb{Z}$, it was proved by Olesen and Pedersen \cite{OlesenPedersen1982}*{Theorem 10.4}, and for separable C*-dynamical systems over $\mathbb{Z}_n$, with $n$ a square-free positive integer, this was proved by Kwa\'{s}niewski and Meyer \cite{KwasniewskiMeyer2018}*{Theorem 2.5}. The general result was obtained recently by Geffen and Ursu \cite{GeffenUrsu2023}*{Corollary 9.6}.

\begin{lem}\label{lem:conditions-for-invariance} Let $(A, G,\alpha)$ be a unital C*-dynamical system with $g\in G \setminus \{e\}$ such that $\alpha_g=\operatorname{Ad}(u_g)$ for a unitary $u\in A$. Let $C_g$ denote the corresponding cyclic subgroup. Suppose that the index $[C_G(g) : C_g]$ of $C_g$ in $C_G(g)$ is finite, $g^{[C_G(g) : C_g]} \ne e$ and $C_G(g)=C_G(g^n)$ for all $n \geq 1$ such that $g^n \ne e$. Then letting $h\coloneqq g^{[C_G(g) : C_g]}$, there is a $C_G(h)$-invariant unitary $u_h \in A$ such that $\alpha_h = \operatorname{Ad}(u_h)$.
\end{lem}

\begin{proof} The proof is very similar to the proofs of the above-mentioned results (see e.g. \cite{GeffenUrsu2023}*{Lemma 9.5}). Let $T \subseteq C_G(g)$ be a left transversal of $C_g$ in $C_G(g)$. Then since $C_g$ has finite index in $C_G(g)$, $T$ is finite. Note that $h \ne e$ by assumption.

For $t \in T$, $\alpha_t(u_g) \in A$ is a unitary satisfying $\operatorname{Ad}(\alpha_t(u_g)) = \alpha_g$.  It follows that $\{\alpha_t(u_g) : t \in T\}$ is a finite set of pairwise commuting unitaries in $A$. Let $u_h := \prod_{t \in T} \alpha_t(u_g) \in A$. Then $u_h$ is a unitary satisfying $\operatorname{Ad}(u_h) = \alpha_h$. The assumption that $C_G(g) = C_G(g^n)$ for all $n \geq 1$ such that $g^n \ne e$ implies that $u_h$ is $C_G(h)$-invariant, because $u_g$ is $C_g$-invariant.
\end{proof}

\begin{prop}\label{prop:conditions-for-proper-outerness}
 $(A, G,\alpha)$ be a C*-dynamical system. For $g \in FC(G) \setminus \{e\}$, let $C_g$ denote the corresponding cyclic subgroup. Suppose that the index $[C_G(g) : C_g]$ of $C_g$ in $C_G(g)$ is finite, $g^{[C_G(g) : C_g]} \ne e$ and $C_G(g)=C_G(g^n)$ for all $n \geq 1$ such that $g^n \ne e$. If $(A,G,\alpha)$ has the regular ideal intersection property, then $\alpha_g$ is properly outer. In particular, this applies if $A \times_\lambda G$ is prime.
\end{prop}

\begin{rem} Note that if $G$ is abelian, then for $g \in G$, $C_G(g) = G$. In particular, $C_G(g)=C_G(g^n)$ for all $n \geq 1$ such that $g^n \ne e$. 
\end{rem}

\begin{proof}
Suppose to the contrary that $\alpha_g$ is not properly outer. We will prove that $(A,G,\alpha)$ does not have the regular ideal intersection property. It follows from \cite{Hamana1982a}*{Proposition 5.1} that there is a largest nonzero central projection $p \in Z(I(A))$ and a unitary $u_g \in I(A)p$ such that $\alpha_g|_{I(A)p} = \operatorname{Ad}(u_g)$. Arguing as in the proof of \cite{Hamana1985}*{Proposition 8.2}, we further obtain that $p$ is $C_G(g)$-invariant. Letting $h\coloneqq g^{[C_G(g) : C_g]}$, and applying Lemma \ref{lem:conditions-for-invariance} with $C_G(g)$ in place of $G$ and $I(A)p$ in place of $I(A)$ yields that $\alpha_h=\operatorname{Ad}(u_h)$ for some $C_G(h)$-invariant unitary $u_h \in I(A)p$. Theorem \ref{thm:regular-ideal-intersection-property} implies that $(A,G,\alpha)$ does not have the regular ideal intersection property. The last statement follows from Theorem \ref{thm:primeness}.
\end{proof}

\begin{rem} \label{rem:iff-for-regular-ideal-intersection-property}
If $(A,G,\alpha)$ is properly outer, then it follows from \cite{KennedySchafhauser2019}*{Corollary~6.7} and \cite{Zarikian2019}*{Theorem~4.6} that $(A,G,\alpha)$ has the ideal intersection property, and in particular the regular ideal intersection property. It follows that if every element in $G \setminus \{e\}$ satisfies the hypotheses of Proposition \ref{prop:conditions-for-proper-outerness}, which in particular implies that $G = FC(G)$, then $(A,G,\alpha)$ has the regular ideal intersection property if and only if it is properly outer.
\end{rem}

The following result, due to Olesen and Pedersen \cite{OlesenPedersen1982}*{Theorem 10.4}, Kwa\'{s}niewski and Meyer \cite{KwasniewskiMeyer2018}*{Theorem 2.5} and Geffen and Ursu \cite{GeffenUrsu2023}*{Corollary 9.6}, was discussed at the beginning of this section. We will now give a short proof utilizing Proposition \ref{prop:conditions-for-proper-outerness}.

\begin{cor} \label{cor:cyclic-groups}
Let $(A,G,\alpha)$ be a C*-dynamical system with $G = \mathbb{Z}$ or $G = \mathbb{Z}_n$, with $n$ a square-free positive integer. Then $(A,G,\alpha)$ has the ideal intersection property if and only if it is properly outer. 
\end{cor}

\begin{proof}
If $G = \mathbb{Z}$, then for $g \in G \setminus \{e\}$,  $g^{[C_G(g) : C_g]} \ne e$. Otherwise, if $G = \mathbb{Z}_n$ for $n$ a square-free positive integer, and $g \in G \setminus \{e\}$, then the order of $g$ is $|C_g|$. If $g^{[C_G(g) : C_g]} = e$, then $|C_g|$ divides $[C_g : C_G(g)]$. Since $[C_G(g) : C_g] |C_g| = |C_G(g)|$ divides $|G| = n$ and $n$ is square-free, it follows that $g^{[C_G(g) : C_g]} \ne e$. In either case, it follows from Proposition \ref{prop:conditions-for-proper-outerness} that $\alpha_g$ is properly outer. Hence $(A,G,\alpha)$ is properly outer. On the other hand, if $(A,G,\alpha)$ is properly outer, then it has the ideal intersection property by \cite{KennedySchafhauser2019}*{Corollary~6.7} and \cite{Zarikian2019}*{Theorem~4.6}.
\end{proof}

\begin{example}[Tarski monster groups]\label{ex:Tarski-monster} Let $G$ be a Tarski monster group corresponding to a prime number $p$, i.e. $G$ is a finitely generated simple group such that every non-trivial subgroup is cyclic of order $p$, and in particular has square-free order. Note that $G$ is necessarily ICC. Olshanskii \cite{Olshanskii1982} proved that $G$ exists if $p > 10^{75}$. By \cite{KalantarKennedy2017}*{Theorem 6.5}, $G$ is C*-simple.

Let $(A,G,\alpha)$ be a C*-dynamical system. For $g \in G \setminus \{e\}$, let $C_g$ denote the cyclic subgroup generated by $g$. It follows from the properties of $G$ that $C_g = C_G(g) \cong \mathbb{Z}_p$. Hence $g$ satisfies the hypotheses of Proposition \ref{prop:conditions-for-proper-outerness}. It follows from Remark \ref{rem:iff-for-regular-ideal-intersection-property} and Theorem \ref{thm:primeness} that $(A,G,\alpha)$ has the regular ideal intersection property if and only if whenever $(A,G,\alpha)$ is sub-induced from a regular C*-dynamical subsystem $(J,H,\beta)$ with $H$ proper, then $(J,H,\beta)$ is properly outer. In particular it follows from Theorem \ref{thm:primeness} that if $(A,G,\alpha)$ is prime, then $A \times_\lambda G$ is prime if and only if whenever $(A,G,\alpha)$ is induced from a regular C*-dynamical subsystem $(J,H,\beta)$ with $H$ proper, then $(J,H,\beta)$ is properly outer.
\end{example}

\subsection{C*-dynamical systems over $\mathrm{PSL}_2(\ZZ)$ and $\mathrm{SL}_2(\ZZ)$}

In this subsection we will consider C*-dynamical systems over the projective special linear group $\mathrm{PSL}_2(\ZZ)$ and the special linear group $\mathrm{SL}_2(\ZZ)$. Using Proposition \ref{prop:conditions-for-proper-outerness}, we will obtain a significant simplification of the conditions in Theorem \ref{thm:primeness}, which characterizes the primality of reduced crossed products, for C*-dynamical systems over these groups. With minor modifications, this result can also be applied to simplify the conditions in Theorem \ref{thm:regular-ideal-intersection-property}, which characterizes the regular ideal intersection property.

The group $\mathrm{PSL}_2(\ZZ)$ is ICC, and in fact C*-simple by \cite{Poznansky2008}. On the other hand, $\mathrm{SL}_2(\ZZ)$ is not ICC, since
\[
FC(\mathrm{SL}_2(\ZZ))=\mathrm{Z}(\mathrm{SL}_2(\ZZ))=\{I_2,-I_2\},
\]
where $I_2$ denotes the $2\times 2$ identity matrix. Because of their restrictive subgroup structure, we will be able to significantly simplify the conditions in Theorem \ref{thm:primeness} for C*-dynamical systems over these groups.

Before considering C*-dynamical systems, we first require two basic results about the structure of centralizers in $\mathrm{PSL}_2(\ZZ)$ and $\mathrm{SL}_2(\ZZ)$. While these results are undoubtedly well known, we were unable to find them stated explicitly in the literature, so we provide the necessary details.

\begin{lem}\label{lem:centralizers-in-PSL} For $r\in \mathrm{PSL}_2(\mathbb{Z}) \setminus \{e\}$, the centralizer $C_{\mathrm{PSL}_2(\ZZ)}(r)$ is isomorphic to one of $\mathbb{Z}$, $\mathbb{Z}_2$ or $\mathbb{Z}_3$. In particular, if $n\geq 1$ and $r^n\neq e$, then $C_{\mathrm{PSL}_2(\mathbb{Z})}(r)=C_{\mathrm{PSL}_2(\mathbb{Z})}(r^n)$.
\end{lem}

\begin{proof} Let $G\coloneqq \mathrm{PSL}_2(\ZZ)$. We will utilize the fact that $G$ is isomorphic to the free product $\ZZ_2*\ZZ_3$. For $r \in G \setminus \{e\}$, it follows from \cite{KarrassSolitar1971}*{Theorem 1} that the centralizer $C_G(r)$ is either conjugate to one of $\mathbb{Z}_2$ or $\mathbb{Z}_3$, or isomorphic to $\mathbb{Z}$.

Suppose now that $n\geq 1$ and $r^n\neq e$. The inclusion $C_G(r) \subseteq C_G(r^n)$ is clear. Also, from above, $C_G(r^n)$ is cyclic, and in particular abelian. Since $r \in C_G(r^n)$, it follows that $C_G(r^n) \subseteq C_G(r)$.
  \end{proof}
    
    \begin{lem}\label{lem:centralizers-in-SL} 
    For $r \in \mathrm{SL}_2(\mathbb{Z}) \setminus Z(SL_2(\mathbb{Z}))$, the centralizer $C_{\mathrm{SL}_2(\mathbb{Z})}(r)$ is isomorphic to one of $\mathbb{Z} \oplus \mathbb{Z}_2$, $\mathbb{Z}_4$ or $\mathbb{Z}_6$. In particular, for a subgroup $H \leq SL_2(\mathbb{Z})$ with $r \in H$, $C_H(r)$ is isomorphic to one of $\mathbb{Z}$, $\mathbb{Z} \oplus \mathbb{Z}_2$, $\mathbb{Z}_3$, $\mathbb{Z}_4$ or $\mathbb{Z}_6$. 
    \end{lem}
    
    \begin{proof} The first statement follows from Lemma \ref{lem:centralizers-in-PSL} since $\mathrm{PSL}_2(\ZZ)$ is isomorphic to $\mathrm{SL}_2(\mathbb{Z})/Z(\mathrm{SL}_2(\mathbb{Z}))$. The second statement follows from the inclusion $C_H(r) \subseteq C_{\mathrm{SL}_2(\mathbb{Z})}(r)$, and
    the fact that the only element in $\mathrm{SL}_2(\mathbb{Z})$ of order $2$ is $-I_2$, so that $C_H(r)\neq \mathbb{Z}_2$.
     \end{proof}

     We now turn our attention to C*-dynamical systems over $\mathrm{PSL}_2(\ZZ)$ and $\mathrm{SL}_2(\ZZ)$. For any group $G$, if $g \in G \setminus \{e\}$ has a cyclic centralizer and $H \leq G$ is a subgroup with $g \in FC(H)$, then $H$ must have a cyclic subgroup of finite index. Lemma \ref{lem:centralizers-in-PSL} implies that the centralizer of any element in $\mathrm{PSL}_2(\ZZ)$ is either finite cyclic, $\mathbb{Z}$ or $\mathbb{Z} \oplus \mathbb{Z}_2$.Similarly,  Lemma \ref{lem:centralizers-in-SL} implies that the centralizer of any non-central element in $\mathrm{SL}_2(\ZZ)$ is cyclic. As a result, we will see that the invariance conditions in Theorem \ref{thm:primeness} are mostly unnecessary for these groups. 

\begin{prop}\label{pro:characterization-for-PSL_2-actions}
Let $(A,\mathrm{PSL}_2(\ZZ),\alpha)$ be a prime C*-dynamical system. The following conditions are equivalent:
\begin{enumerate}
    \item $A\times_\lambda \mathrm{PSL}_2(\ZZ) $ is prime;

    \item Whenever $(A,\mathrm{PSL}_2(\ZZ),\alpha)$ is induced from a regular C*-dynamical subsystem $(J,H,\beta)$ with $r \in H \setminus \{e\}$ such that $\beta_r$ is not properly outer, then the $H$-conjugacy class $\operatorname{Cl}_H(r)$ of $r$ is infinite.
\end{enumerate}
\end{prop}

\begin{proof} First suppose that (1) holds, i.e. that $A\times_\lambda \mathrm{PSL}_2(\ZZ)$ is prime. Then by Theorem \ref{thm:primeness}, $(A,\mathrm{PSL}_2(\ZZ),\alpha)$ has the regular ideal intersection property. Suppose in addition that $(A,\mathrm{PSL}_2(\ZZ),\alpha)$ is induced from a regular C*-dynamical subsystem $(J,H,\beta)$. It follows from Theorem \ref{thm:imprimitivity-reduced-crossed-product} that $J \times_\lambda H$ is also prime. For $r \in FC(H) \setminus \{e\}$, Lemma \ref{lem:centralizers-in-PSL} implies that $C_H(r) = C_H(r^n)$ for all $n \geq 1$ such that $r^n \ne e$, and that $C_H(r)$ is isomorphic to one of $\mathbb{Z}$, $\mathbb{Z}_2$ or $\mathbb{Z}_3$. In each case, the cyclic subgroup $C_r$ generated by $r$ satisfies $[C_H(r) : C_r] < \infty$ and $r^{[C_H(r) : C_r]} \ne e$.  Hence the assumptions of Proposition \ref{prop:conditions-for-proper-outerness} are satisfied, and it follows that the action of $r$ on $J$ must be properly outer, giving (2). 

Conversely, suppose that (2) holds. Suppose that $(A,\mathrm{PSL}_2(\ZZ),\alpha)$ is induced from a regular C*-dynamical subsystem $(J, H, \beta)$ with $r \in FC(H) \setminus \{e\}$. By assumption, $\beta_r$ is properly outer, so in particular its extension to $I(J)$ is not inner. It follows from the equivalence between Conditions (1) and (5) of Theorem \ref{thm:primeness} that (1) holds, i.e. that $A\times_\lambda \mathrm{PSL}_2(\ZZ) $ is prime.
    \end{proof}

    \begin{prop}\label{pro:characterization-for-SL_2-actions}
    Let $(A,\mathrm{SL}_2(\ZZ),\alpha)$ be a prime C*-dynamical system. The following conditions are equivalent:
\begin{enumerate}
    \item $A\times_\lambda \mathrm{SL}_2(\ZZ) $ is prime;

    \item There is no $\mathrm{SL}_2(\ZZ)$-invariant essential hereditary C*-subalgebra $B \subseteq A$ such that $\alpha_{-I_2}|_B = \operatorname{exp}(\delta)$ for a $\mathrm{SL}_2(\ZZ)$-commuting *-derivation $\delta$ of $B$, and whenever $(A,\mathrm{SL}_2(\mathbb{Z}),\alpha)$ is induced from a regular C*-dynamical subsystem $(J,H,\beta)$ with $r \in H \setminus Z(\mathrm{SL}_2(\mathbb{Z}))$ such that $\beta_r$ is not properly outer, then the $H$-conjugacy class $\operatorname{Cl}_H(r)$ of $r$ is infinite.

    \item There is no $\mathrm{SL}_2(\ZZ)$-invariant unitary $u\in I(A)$ such that $\alpha_{-I_2}=\operatorname{Ad}(u)$ on $I(A)$, and whenever $(A,\mathrm{SL}_2(\mathbb{Z}),\alpha)$ is induced from a regular C*-dynamical subsystem $(J,H,\beta)$ with $r \in H \setminus Z(\mathrm{SL}_2(\mathbb{Z}))$ such that $\beta_r$ is not properly outer, then the $H$-conjugacy class $\operatorname{Cl}_H(r)$ of $r$ is infinite.
\end{enumerate}
\end{prop}

\begin{proof} We first prove the equivalence between (1) and (3). Suppose that (1) holds, i.e. that $A\times_\lambda \mathrm{SL}_2(\ZZ) $ is prime. Since $-I_2 \in Z(\mathrm{SL}_2(\ZZ))$, the equivalence between Conditions (1) and (5) of Theorem \ref{thm:primeness} implies the first part of the statement in (2). For the second part, suppose that $(A,\mathrm{SL}_2(\ZZ),\alpha)$ is induced from a regular C*-dynamical subsystem $(J,H,\beta)$ with $r\in FC(H) \setminus Z(\mathrm{SL}_2(\mathbb{Z}))$.

Lemma \ref{lem:centralizers-in-SL} implies that $C_H(r) = C_H(r^n)$ for all $n \geq 1$ such that $r^n \ne e$, and that $C_H(r)$ is isomorphic to one of $\mathbb{Z}$, $\mathbb{Z} \oplus \mathbb{Z}_2$, $\mathbb{Z}_3$, $\mathbb{Z}_4$ or $\mathbb{Z}_6$. In each case, the cyclic subgroup $C_r$ generated by $r$ satisfies $[C_H(r) : C_r] < \infty$ and $r^{[C_H(r) : C_r]} \ne e$. Hence the assumptions of Proposition \ref{prop:conditions-for-proper-outerness} are satisfied, and it follows that the action of $r$ on $J$ must be properly outer, giving (3). 

 Conversely, suppose that (3) holds. Suppose that $(A,\mathrm{SL}_2(\mathbb{Z}),\alpha$ is induced from a regular C*-dynamical subsystem $(J,H,\beta)$. For $r\in FC(H) \setminus Z(\mathrm{SL}_2(\mathbb{Z}))$, $\beta_r$ is properly outer by assumption, so its extension to $I(J)$ is not inner. 

 It remains to consider the element $r := -I_2 \in \mathrm{SL}_2(\mathbb{Z})$. Note that $r \in Z(H)$, so $C_H(r) = H$. Suppose for the sake of contradiction that there is an $H$-invariant unitary $u_r \in I(J)$ such that $\beta_r = \operatorname{Ad}(u_r)$. We claim there is an $\mathrm{SL}_2(\mathbb{Z})$-invariant unitary $u\in I(A)$ such that $\alpha_{-I_2}=\operatorname{Ad}(u)$. 

 To see this, let $T$ be a left transversal of $H$ in $\mathrm{SL}_2(\mathbb{Z})$. For $g \in T$ and $b \in I(A)\alpha_g(p)$,
 \begin{align*}
     \alpha_r(b)&=\alpha_r(\alpha_g(\alpha_{g^{-1}}(b)))=\alpha_g(\alpha_r(\alpha_{g^{-1}}(b)))\\&=\alpha_g(u_r \alpha_{g^{-1}}(b) u_r^*)=\alpha_g(u_r)b\alpha_{g}(u_r)^*. 
 \end{align*}
Hence $\alpha_r|_{I(A)\alpha_g(p)} = \operatorname{Ad}(\alpha_g(u_r))$. Since the projections $\{\alpha_g(p) : g \in T\}$ are mutually orthogonal, for each finite subset $F \subseteq T$, the corresponding sum $v_F := \sum_{g \in F} \alpha_g(u_r) \in I(A)$ is a partial isometry. Taking the order limit over $F$ in $I(A)$, we obtain an $\mathrm{SL}_2(\mathbb{Z})$-invariant unitary $v \in I(A)$ such that $\alpha_r = \operatorname{Ad}(v)$, contradicting the first part of (2). It now follows from the equivalence between Conditions (1) and (5) of Theorem \ref{thm:primeness} that (1) holds, i.e. that $A \times_\lambda \mathrm{SL}_2(\mathbb{Z})$ is prime. The equivalence between (2) and (3) also follows from Theorem \ref{thm:primeness}.
\end{proof}

\begin{rem}
It is straightforward to verify that in the separable setting, these conditions are also equivalent to a simplified version of Condition (7) in Theorem \ref{thm:primeness}. 
\end{rem}

\subsection{Free products of cyclic groups}
In this subsection we will consider C*-dynamical systems over finite free products of cyclic groups of square-free order. A straightforward extension of the arguments in Lemma \ref{lem:centralizers-in-PSL} and Proposition \ref{pro:characterization-for-PSL_2-actions}, which utilize Proposition \ref{prop:conditions-for-proper-outerness}, will yield a simplification of the conditions in Theorem \ref{thm:primeness}, which characterizes the primality of reduced crossed products, for C*-dynamical systems over these groups. While this result will still require the consideration of regular C*-dynamical subsystems over subgroups with non-trivial FC-center, we will be able to give a complete description of subgroups with this property (see Remark \ref{rem:FC-subgroups-free-products}).

\begin{prop}\label{pro:characterization-for-free-products} 
For $m \geq 1$, let $G = G_1 \ast \cdots \ast G_m $ be a finite free product of cyclic groups of square-free order, i.e. for $1 \leq i \leq m$, $G_i = \mathbb{Z}$ or $G_i = \mathbb{Z}_{n_i}$ for $n_i$ a square-free positive integer. Let $(A,G,\alpha)$ be a prime C*-dynamical system. The following are equivalent:
\begin{enumerate}
    \item $A\times_\lambda G $ is prime;

    \item Whenever $(A,G,\alpha)$ is induced from a regular C*-dynamical subsystem $(J,H,\beta)$ with $r \in FC(H) \setminus \{e\}$, then $\beta_r$ is properly outer.
\end{enumerate}
\end{prop}
\begin{proof}
It follows from \cite{KarrassSolitar1971}*{Theorem 1} that for $r \in G \setminus \{e\}$, the centralizer $C_G(r)$ is either conjugate to a factor $G_i$ or isomorphic to $\mathbb{Z}$. The equivalence between (1) and (2) now follows as in the proof of Proposition \ref{pro:characterization-for-PSL_2-actions}. 
\end{proof}

\begin{example}[Free groups] Proposition \ref{pro:characterization-for-free-products} applies in particular to the free group $\mathbb{F}_n \cong \mathbb{Z}^{\ast n}$ on $n \geq 1$ generators. Let $(A, \mathbb{F}_n, \alpha)$ be a prime C*-dynamical system. If $A\times_\lambda G $ is not prime, then by Theorem \ref{thm:primeness}, $(A,\mathbb{F}_n,\alpha)$ is induced from a regular C*-dynamical subsystem $(J,H,\beta)$ such that $FC(H)$ is non-trivial. By the Nielsen–Schreier theorem, $H$ is free, so $H$ must be isomorphic to $\mathbb{Z}$. 
\end{example}

\begin{example}[Infinite dihedral group] The infinite dihedral group $D_\infty=\ZZ_2*\ZZ_2$ is FC-hypercentral but not FC. This follows from the alternative presentation $D_\infty \cong \langle x,y : x^2 = e,\ xyx = y^{-1} \rangle$, from which we see that $D_\infty=\{x^\varepsilon y^n:\varepsilon\in\{0,1\}, n\in \ZZ\}$, and hence that $FC(D_\infty)$ is the cyclic group of order $2$ generated by $y$.

Let $(A,D_\infty,\alpha)$ be a C*-dynamical system. Since $D_\infty$ is FC-hypercentral from above, we can apply Theorem \ref{thm:FC-hypercentral-ideal-intersection-property}. However, by Proposition \ref{pro:characterization-for-free-products}, the invariance condition in Conditions (4), (5) and (7) in Theorem \ref{thm:FC-hypercentral-ideal-intersection-property} are unnecessary. It follows that $(A,D_\infty,\alpha)$ has the ideal intersection property if and only if whenever $(A,D_\infty,\alpha)$ is sub-induced from a regular C*-dynamical system $(J,H,\beta)$ with $r \in FC(H) \setminus \{e\}$, then $\beta_r$ is properly outer. 

If, in addition, $(A,D_\infty,\alpha)$ is minimal, then $A \times_\lambda D_\infty$ is simple if and only if $\alpha_{y^n}$ is properly outer for all $n \geq 1$, where $y \in D_\infty$ denotes the second generator in the alternate presentation of $D_\infty$ above.
\end{example}

\begin{rem}\label{rem:FC-subgroups-free-products} 
It follows from Condition (2) of Proposition \ref{pro:characterization-for-free-products} that in order to verify the primality of the reduced crossed product of a C*-dynamical system over a finite free product $G$ of cyclic groups, it is only necessary to consider regular C*-dynamical systems over subgroups $H$ with non-trivial FC center. We will now give a complete description of subgroups $H \leq G$ with this property. 

We claim that if $H \leq G$ has non-trivial FC-center, then either $H$ is conjugate to a subgroup of one of the factors (and hence is cyclic), $H$ is isomorphic to the infinite cylic group $\mathbb{Z}$ or $H$ is isomorphic to the infinite dihedral group $\mathbb{Z}_2 \ast \mathbb{Z}_2$. To see this, note that by the Kurosh subgroup theorem, $H$ is itself a (potentially infinite) free product of cyclic groups in which each factor is isomorphic to either $\mathbb{Z}$ or a subgroup of one of the factors of $G$. The claim now follows from the fact that the free product of two non-trivial groups has non-trivial FC-center if and only if each factor has order $2$.
\end{rem}

\end{document}